\documentclass{amsart}

\usepackage{amssymb}
\usepackage{amsmath}
\usepackage{amsthm}
\usepackage{amsfonts}

\usepackage{a4wide}
\usepackage{longtable}

\newtheorem{theorem}{Theorem}[section]
\newtheorem{proposition}{Proposition}[section]
\newtheorem{corollary}{Corollary}[section]

\theoremstyle{definition}
\newtheorem{definition}{Definition}[section]
\newtheorem{remark}{Remark}[section]

\numberwithin{equation}{section}

\begin{document}

\title{Annihilator ideals of two-generated metabelian \(p\)-groups}

\author{Daniel C. Mayer}
\address{Naglergasse 53\\8010 Graz\\Austria}
\email{algebraic.number.theory@algebra.at}
\urladdr{http://www.algebra.at}
\thanks{Research supported by the Austrian Science Fund (FWF): P 26008-N25}

\subjclass[2000]{Primary 13F20, 13P10; Secondary 20D15, 20F12, 20F14}
\keywords{Ideals of bivariate polynomials with integer coefficients, residue class rings,
metabelian \(p\)-groups with two generators, action on the commutator subgroup, commutator calculus, central series}

\date{March 30, 2016}

\dedicatory{Dedicated to the memory of Otto Schreier}

\begin{abstract}
For a metabelian \(p\)-group \(G=\langle x,y\rangle\) with two generators \(x\) and \(y\),
the annihilator \(\mathfrak{A}\,\unlhd\,\mathbb{Z}\lbrack X,Y\rbrack\)
of the main commutator \(\lbrack y,x\rbrack\) of \(G\),
as an ideal of bivariate polynomials with integer coefficients,
is determined by means of a presentation for \(G\).
Furtw\"angler's isomorphism of the additive group underlying the residue class ring
\(\mathbb{Z}\lbrack X,Y\rbrack/\mathfrak{A}\)
of \(\mathbb{Z}\lbrack X,Y\rbrack\) modulo the annihilator \(\mathfrak{A}\)
to the commutator subgroup \(G^\prime\) of \(G\)
admits the calculation of the abelian type of \(G^\prime\).
\end{abstract}

\maketitle



\section{Introduction}
\label{s:Intro}

\noindent
Let \(p\ge 2\) be a prime number
and assume that \(G\) is a finite metabelian \(p\)-group
with abelianization \(G/G^\prime\) of type \((p,p)\).
Then \(G=\langle x,y\rangle\) has two generators \(x\) and \(y\)
and we call \(s_2:=\lbrack y,x\rbrack\in G^\prime\)
the \textit{main commutator} of \(G\).
There is an action \(G\times G^\prime\to G^\prime\),
\((g,t)\mapsto t^g:=g^{-1}tg\),
of \(G\) on the commutator subgroup \(G^\prime\)
via inner automorphisms,
since \(t^{gh}=(gh)^{-1}tgh=h^{-1}g^{-1}tgh=(t^g)^h\), \(t^1=t\),
and \(G^\prime\) is a characteristic subgroup of \(G\).
This action can be extended to the integral group algebra
\(\mathbb{Z}\lbrack G\rbrack\)
by putting \(t^{mg+nh}:=(t^g)^m\cdot (t^h)^n\),
for \(m,n\in\mathbb{Z}\), \(g,h\in G\).
Based on this extended action,
Furtw\"angler proved the following theorem
\cite{Fw},
which is crucial for the present article.


\begin{theorem}
\label{thm:Furtwaengler}
(Ph. Furtw\"angler, \(1929\))\quad
The composite map

\begin{equation}
\label{eqn:Furtwaengler}
\begin{aligned}
\Psi:\quad & \mathbb{Z}\lbrack X,Y\rbrack \buildrel \psi \over{\longrightarrow} \mathbb{Z}\lbrack G\rbrack \longrightarrow G^\prime,        \\
           & f(X,Y)                       \mapsto         f(x-1,y-1)                 \mapsto         s_2^{f(x-1,y-1)}
\end{aligned}
\end{equation}

\noindent
is an epimorphism
from the underlying additive group of the ring \(\mathbb{Z}\lbrack X,Y\rbrack\) of bivariate polynomials
in two indeterminates \(X,Y\) onto the (multiplicative) commutator subgroup \(G^\prime\) of \(G\).
\end{theorem}

\begin{remark}
\label{rmk:SymbolicPowers}
Note that
\(t^{x-1}=t^{-1+x}=t^{-1}x^{-1}tx=\lbrack t,x\rbrack\), and similarly \(t^{y-1}=\lbrack t,y\rbrack\),
since \(G^\prime\) is abelian.
Consequently, if \(f(X,Y)=\sum_{i\ge 0}\sum_{j\ge 0}\,n_{i,j}X^iY^j\in\mathbb{Z}\lbrack X,Y\rbrack\), then the \textit{symbolic power}

\begin{equation}
\label{eqn:SymbolicPowers}
s_2^{f(x-1,y-1)}=\prod_{i\ge 0}\prod_{j\ge 0}\,\left(s_2^{(x-1)^i(y-1)^j}\right)^{n_{i,j}}
\end{equation}

\noindent
is a finite product of (usual) powers of \textit{iterated commutators}.
\end{remark}

In
\cite{Fw},
Furtw\"angler has in fact proved a generalization of this theorem
for a finite metabelian \(p\)-group \(G=\langle x_1,\ldots,x_n\rangle\) with generator rank \(d_1(G)=n\)
and its \(\binom{n}{2}\) main commutators \(t_{j,k}:=\lbrack x_k,x_j\rbrack\), \(1\le j<k\le n\).
Furtw\"angler's theorem states that the commutator subgroup \(G^\prime\) consists
of all products of symbolic powers
\(\left\lbrace\prod_{1\le j<k\le n}\,t_{j,k}^{f_{j,k}(x_1-1,\ldots,x_n-1)}\mid f_{j,k}(X_1,\ldots,X_n)\in\mathbb{Z}\lbrack X_1,\ldots,X_n\rbrack\right\rbrace\).


\begin{definition}
\label{dfn:Annihilator}
The kernel of the epimorphism \(\Psi\),

\begin{equation}
\label{eqn:Annihilator}
\mathfrak{A}=\mathfrak{A}(G):=\ker(\Psi)=\left\lbrace f(X,Y)\in\mathbb{Z}\lbrack X,Y\rbrack\mid s_2^{f(x-1,y-1)}=1\right\rbrace,
\end{equation}

\noindent
is an ideal of the ring \(\mathbb{Z}\lbrack X,Y\rbrack\)
which is called the \textit{annihilator} of the group \(G=\langle x,y\rangle\), more precisely, of the commutator subgroup \(G^\prime\),
or the \textit{symbolic order} of the main commutator \(s_2=\lbrack y,x\rbrack\),
with respect to powers with symbolic exponents in \(\mathbb{Z}\lbrack G\rbrack\).
\end{definition}


An application of the isomorphism theorem yields:

\begin{corollary}
\label{cor:CommutatorSubgroup}
The additive group underlying the residue class ring of \(\mathbb{Z}\lbrack X,Y\rbrack\) modulo \(\mathfrak{A}\),

\begin{equation}
\label{eqn:CommutatorSubgroup}
\mathbb{Z}\lbrack X,Y\rbrack/\mathfrak{A}=\mathbb{Z}\lbrack X,Y\rbrack/\ker(\Psi)\simeq\mathrm{im}(\Psi)=G^\prime,
\end{equation}

\noindent
is isomorphic to the (multiplicative) commutator subgroup \(G^\prime\) of \(G\).
\end{corollary}


The layout of this article is as follows.
We focus on three infinite classes of groups
for which parametrized presentations are known.
Required background is provided in \S\ 
\ref{ss:MaxClass3}
for metabelian \(3\)-groups of coclass \(1\), in \S\
\ref{ss:NonMaxClass3}
for metabelian \(3\)-groups \(G\) of coclass \(\mathrm{cc}(G)\ge 2\) with abelianization \(G/G^\prime\) of type \((3,3)\), and in \S\
\ref{ss:MaxClassP}
for metabelian \(p\)-groups of coclass \(1\), for an arbitrary prime \(p\ge 2\).
The aim of \S\S\
\ref{s:MaxSmbOrd3},
\ref{s:LowSmbOrd},
\ref{s:MaxSmbOrdP}
is to determine the annihilator \(\mathfrak{A}\) of these metabelian \(p\)-groups.
In \S\
\ref{s:StrComSbg},
finally, the structure of the commutator subgroup \(G^\prime\) of \(G\) is calculated
with the aid of a suitable basis of the additive group underlying the residue class ring
\(\mathbb{Z}\lbrack X,Y\rbrack/\mathfrak{A}\).



\section{Presentations of metabelian \(p\)-groups}
\label{s:Presentations}

\noindent
For identifying those bivariate polynomials
\(f(X,Y)\in\mathbb{Z}\langle X,Y\rangle\)
with integer coefficients
which generate the annihilator ideal \(\mathfrak{A}\)
of the main commutator \(s_2=\lbrack y,x\rbrack\)
of a two-generated metabelian \(p\)-group \(G=\langle x,y\rangle\)
we need a presentation of the group.


Generally,
let \(G=\langle x,y\rangle\) be a metabelian \(p\)-group with two generators.
Assume that \(G\) is of order \(\lvert G\rvert=p^n\), where \(n\) is the \textit{logarithmic order},
of \textit{nilpotency class} \(c=\mathrm{cl}(G)=m-1\), where \(m\) is the \textit{index of nilpotency},
and of \textit{coclass} \(r=\mathrm{cc}(G)=n-c=n-m+1\).
Denote the \textit{lower central series} of \(G\) by

\begin{equation}
\label{eqn:LowerCentral}
G=\gamma_1(G)>\gamma_2(G)>\gamma_3(G)>\ldots>\gamma_{m-1}(G)>\gamma_m(G)=1,
\end{equation}

\noindent
where \(\gamma_j(G)=\lbrack\gamma_{j-1}(G),G\rbrack\), for \(j\ge 2\),
and the \textit{upper central series} of \(G\) by

\begin{equation}
\label{eqn:UpperCentral}
1=\zeta_0(G)<\zeta_1(G)<\zeta_2(G)<\ldots<\zeta_{c-1}(G)<\zeta_c(G)=G,
\end{equation}

\noindent
where \(\zeta_j(G)/\zeta_{j-1}(G)=\mathrm{Centre}(G/\zeta_{j-1}(G))\), for \(j\ge 1\).



\subsection{\(3\)-groups of maximal class}
\label{ss:MaxClass3}

\noindent
Let \(G\) be a metabelian \(3\)-group of coclass \(r=\mathrm{cc}(G)=1\),
where \(3\le m=n\), i.e., \(G\) is non-abelian.
Denote the \textit{two-step centralizer} of \(\gamma_2(G)/\gamma_4(G)\) by

\begin{equation}
\label{eqn:TwoStepCentralizerMax3}
\begin{aligned}
\chi_2(G)             &:= \left\lbrace g\in G\mid\lbrack u,g\rbrack\in\gamma_4(G), \text{ for all } u\in\gamma_2(G)\right\rbrace, \text{ that is,}\\
\chi_2(G)/\gamma_4(G) &= \mathrm{Centralizer}_{G/\gamma_4(G)}(\gamma_2(G)/\gamma_4(G)),
\end{aligned}
\end{equation}

\noindent
whence \(\gamma_2(G)<\chi_2(G)<G\) \(\Longleftrightarrow\) \(m\ge 4\),
and select \textit{normalized generators} of \(G=\langle x,y\rangle\) such that

\begin{equation}
\label{eqn:GeneratorsMax3}
x\in G\setminus\chi_2(G), \text{ if } m\ge 4, \quad y\in\chi_2(G)\setminus\gamma_2(G).
\end{equation}

\noindent
Starting with the main commutator \(s_2:=\lbrack y,x\rbrack\in\gamma_2(G)\)
define the \textit{higher iterated commutators} recursively by
\(s_j:=\lbrack s_{j-1},x\rbrack\in\gamma_j(G)\), for \(j\ge 3\),
and observe that they can be expressed as

\begin{equation}
\label{eqn:SymbolicPowersMax3}
s_j=s_2^{(x-1)^{j-2}}, \text{ for } j\ge 2,
\end{equation}

\noindent
by \textit{symbolic powers}, as explained in Remark
\ref{rmk:SymbolicPowers}.


The general theory of \(p\)-groups of maximal class can be found in Huppert
\cite[\S\ 14, p. 361]{Hp}
and Berkovich
\cite[\S\ 9, p. 114]{Bv}.
In the isomorphism class of a metabelian \(3\)-group \(G\) of maximal class,
there exists a representative \(G_\gamma^m(\beta,\alpha)\)
whose normalized generators satisfy the following relations,
according to Blackburn
\cite{Bl},
Miech
\cite{Mi1,Mi2},
and Nebelung
\cite[\S\ 3.2, p. 58]{Ne1}:

\begin{itemize}
\item
parametrized \textit{power relations} with parameters \(-1\le\alpha,\beta\le 1\),

\begin{equation}
\label{eqn:PowerRelMax3}
\begin{aligned}
x^3                   &= s_{m-1}^\alpha\in\zeta_1(G), \\
y^3s_2^3s_3           &= s_{m-1}^\beta\in\zeta_1(G), \\
s_j^3s_{j+1}^3s_{j+2} &= 1, \text{ for all } j\ge 2,
\end{aligned}
\end{equation}

\item
parametrized \textit{commutator relations} with parameter \(-1\le\gamma\le 1\),

\begin{equation}
\label{eqn:CommutatorRelMax3}
\lbrack s_2,y\rbrack=s_{m-1}^\gamma\in\zeta_1(G), \quad \text{ but } \lbrack s_j,y\rbrack=1, \text{ for all } j\ge 3,
\end{equation}

\item
parametrized \textit{nilpotency relations} with parameter \(m\ge 3\) (the index of nilpotency),

\begin{equation}
\label{eqn:NilpotencyRelMax3}
s_j\ne 1, \text{ for } j\le m-1, \quad \text{ but } s_j=1, \text{ for } j\ge m,
\end{equation}

\noindent
expressing the polycyclic structure
\(\gamma_j(G)=\langle s_j,\gamma_{j+1}(G)\rangle\),
\(\gamma_j(G)/\gamma_{j+1}(G)\simeq C_3\), for \(j\ge 2\) and \(j\le m-1\),
with \textit{cyclic factors} (CF) of the lower central series,
which coincides here with the reverse upper central series, \(\zeta_j(G)=\gamma_{m-j}(G)\), for \(0\le j\le m-1\),

\item
and (trivial) \textit{metabelian relations} within \(G^\prime\),

\begin{equation}
\label{eqn:MetabelianRelMax3}
\lbrack s_i,s_j\rbrack=1, \text{ for all } i,j\ge 2.
\end{equation}

\end{itemize}


\noindent
In the case of an index of nilpotency \(m\ge 4\), the commutator relation for \(s_2\) in Formula
\eqref{eqn:CommutatorRelMax3}
explicitly describes the properties of the two-step centralizer
\(\chi_2(G)=\langle y,\gamma_2(G)\rangle\),
which are implicitly postulated by Formula
\eqref{eqn:TwoStepCentralizerMax3}:

\begin{equation}
\label{eqn:DefectMax3}
\lbrack\gamma_2(G),\chi_2(G)\rbrack=\gamma_{m-k}(G)=\zeta_k(G),
\end{equation}

\noindent
where \(0\le k=k(G)\le 1\) denotes the \textit{defect of commutativity} of \(G\).
It is closely related to the parameter \(\gamma\), since we have
\(k=0\) \(\Longleftrightarrow\) \(\gamma=0\), and
\(k=1\) \(\Longleftrightarrow\) \(\gamma=\pm 1\).


\begin{remark}
\label{rmk:ParametrizedPresentation}
Note that we have used Nebelung's parameters \(\alpha,\beta,\gamma\)
in the presentation of \(G\), which are denoted
\(\delta,\gamma,-\beta\) by Blackburn
\cite[p. 82, (33), and p. 84, (36), (37)]{Bl},
\(w,z,-a(m-1)\) by Miech
\cite{Mi1},
and \(w,z,a(m-1)\) by Miech
\cite{Mi2},
respectively, in the same order.
\end{remark}



\subsection{\(3\)-groups of non-maximal class}
\label{ss:NonMaxClass3}

\noindent
Let \(G\) be a metabelian \(3\)-group of coclass \(r=\mathrm{cc}(G)\ge 2\)
with abelianization \(G/G^\prime\) of type \((3,3)\).
As before, assume that \(G\) has order \(\lvert G\rvert=3^n\),
nilpotency class \(c=\mathrm{cl}(G)=m-1\), and index of nilpotency \(m\).
Then \(G\) cannot be a CF-group and
must have at least one bicyclic factor
\(\gamma_3(G)/\gamma_4(G)\simeq C_3\times C_3\)
of the lower central series, as will be explained in section \S\
\ref{ss:Exceptions}.
The CF-\textit{invariant} \(e=e(G)\) characterizes cyclic factors of the lower central series,

\begin{equation}
\label{eqn:CFInvariant}
e+1:=\min\left\lbrace\ j\ge 3\ \mid\ \left(\gamma_j(G):\gamma_{j+1}(G)\right)\le 3\ \right\rbrace.
\end{equation}

\noindent
In particular, \(e=2\) \(\Longleftrightarrow\) \(G\) is a CF-group. The resultant structure of the lower central series,

\begin{equation}
\label{eqn:BFandCFofLowerCentral}
\lvert G\rvert=
\underbrace{(G:1)}_{=3^n}=
\underbrace{(G:\gamma_2)}_{=3^2}\cdot
\underbrace{(\gamma_2:\gamma_3)}_{=3}\cdot
\underbrace{(\gamma_3:\gamma_4)\cdots(\gamma_e:\gamma_{e+1})}_{=(3^2)^{e-2}}\cdot
\underbrace{(\gamma_{e+1}:\gamma_{e+2})\cdots(\gamma_c:\gamma_{c+1})}_{=3^{c-e}},
\end{equation}

\noindent
establishes some relations between invariants under isomorphism.
Firstly, we obtain the equation \(c+r=n=2+1+2(e-2)+c-e=c+e-1\),
which connects the coclass \(r\) with the CF-invariant \(e\) by \(r=e-1\).
Secondly, the nilpotency index \(m\) determines an upper bound for the CF-invariant \(e\le c=m-1\),
and thus also for the logarithmic order \(n=c+r=c+e-1\le m-1+m-1-1=2m-3\).
Conversely, the coclass yields a lower bound for the class \(c\ge e=r+1\),
in particular, \(c\ge 3\) for \(r=2\).
In summary, we have the following statements for a group of non-maximal class:

\begin{equation}
\label{eqn:ConnectingInvariants}
r=e-1, \quad e \le m-1, \quad 4\le m<n\le 2m-3.
\end{equation}


Next, we must generalize the two-step centralizer \(\chi_2(G)\) for groups of non-maximal class.
Denote the \textit{centralizer of the two-step factor group} \(\gamma_j(G)/\gamma_{j+2}(G)\)
of the lower central series by

\begin{equation}
\label{eqn:TwoStepCentralizerLow}
\begin{aligned}
\chi_j(G)                 &:= \left\lbrace g\in G\mid\lbrack u,g\rbrack\in\gamma_{j+2}(G), \text{ for all } u\in\gamma_j(G)\right\rbrace, \text{ that is,}\\
\chi_j(G)/\gamma_{j+2}(G) &= \mathrm{Centralizer}_{G/\gamma_{j+2}(G)}(\gamma_j(G)/\gamma_{j+2}(G)), \text{ for each } j\ge 2.
\end{aligned}
\end{equation}

\noindent
\(\chi_j(G)\) is the biggest subgroup of \(G\) with the property
\(\lbrack\chi_j(G),\gamma_j(G)\rbrack\le\gamma_{j+2}(G)\).
These two-step centralizers form an ascending chain,
\(G^\prime\le\chi_2(G)\le\ldots\le\chi_{m-2}(G)<\chi_{m-1}(G)=G\),
of characteristic subgroups of \(G\)
which contain the commutator subgroup \(G^\prime\).
The centralizer \(\chi_j(G)\) coincides with \(G\) if and only if \(j\ge m-1\).
We characterize the smallest two-step centralizer
different from the commutator group
by the \textit{centralizer invariant} \(s=s(G)\):

\begin{equation}
\label{eqn:CentralizerInvariant}
s:=\min\left\lbrace 2\le j\le m-1\mid\chi_j(G)>G^\prime\right\rbrace.
\end{equation}


\noindent
According to Nebelung
\cite[p. 57]{Ne1},
\textit{normalized generators} of \(G=\langle x,y\rangle\) can be selected such that
the first bicyclic factor of the lower central series is generated by their third powers,

\begin{equation}
\label{eqn:GeneratorsLow}
\gamma_3(G)=\langle y^3,x^3,\gamma_4(G)\rangle, \quad \text{ and } x\in G\setminus\chi_s(G), \text{ if } s<m-1, \quad y\in\chi_s(G)\setminus G^\prime.
\end{equation}

\noindent
Denote by
\(s_2:=t_2:=\lbrack y,x\rbrack\in\gamma_2(G)=G^\prime\)
the main commutator of \(G\),
and let \textit{higher iterated commutators} be declared recursively by
\(s_j:=\lbrack s_{j-1},x\rbrack\), \(t_j:=\lbrack t_{j-1},y\rbrack\in\gamma_j(G)\),
for \(j\ge 3\).
Observe that, as explained in Remark
\ref{rmk:SymbolicPowers},
these commutators can be expressed as \textit{symbolic powers}

\begin{equation}
\label{eqn:SymbolicPowersLow3}
s_j=s_2^{(x-1)^{j-2}}, \quad t_j=s_2^{(y-1)^{j-2}}, \text{ for } j\ge 2.
\end{equation}

\noindent
Starting with the powers \(\sigma_3:=y^3\), \(\tau_3:=x^3\in\gamma_3(G)\), let
\(\sigma_j:=\lbrack\sigma_{j-1},x\rbrack\), \(\tau_j:=\lbrack\tau_{j-1},y\rbrack\in\gamma_j(G)\),
for \(j\ge 4\),
and put \(\Sigma_j:=\langle\sigma_j,\ldots,\sigma_{m-1}\rangle\),
\(T_j:=\langle\tau_j,\ldots\tau_{e+1}\rangle\), for \(j\ge 3\).

Finally, let the \textit{defect of commutativity} \(0\le k=k(G)\le 1\) of \(G\) be defined by

\begin{equation}
\label{eqn:DefectLow}
\lbrack\chi_s(G),\gamma_e(G)\rbrack=\gamma_{m-k}(G).
\end{equation}


One of the principal results of Nebelung's thesis
\cite[p. 94]{Ne1},
which is also briefly summarized in
\cite[p. 456]{Ma3},
is the following parametrized presentation.
In the isomorphism class of a metabelian \(3\)-group \(G\) of non-maximal class
with abelianization \(G/G^\prime\) of type \((3,3)\),
there exists a representative \(G_\rho^{m,n}(\alpha,\beta,\gamma,\delta)\)
whose normalized generators satisfy the following relations:

\begin{itemize}

\item
\textit{basic power relations},

\begin{equation}
\label{eqn:LowPowerRel}
\begin{aligned}
s_i^3s_{i+1}^3s_{i+2}=1,                &\quad t_i^3t_{i+1}^3t_{i+2}=1, \quad \text{ for } i\ge 3, \\
\sigma_j^3\sigma_{j+1}^3\sigma_{j+2}=1, &\quad \tau_j^3\tau_{j+1}^3\tau_{j+2}=1, \quad \text{ for } j\ge 3,
\end{aligned}
\end{equation}

\item
\textit{supplementary power relations},

\begin{equation}
\label{eqn:LowSupplPowerRel}
s_2^3s_3^3s_4=\tau_4^{-1}, \quad t_2^3t_3^3t_4=\sigma_4,
\end{equation}

\item
\textit{commutator relations},

\begin{equation}
\label{eqn:LowCommutatorRel}
\lbrack\sigma_j,y\rbrack=1, \quad \lbrack\tau_j,x\rbrack=1, \text{ for } j\ge 3,
\end{equation}

\noindent
expressing the structure \(\gamma_j(G)=\langle\sigma_j,\tau_j,\gamma_{j+1}(G)\rangle=\Sigma_j T_j\)
with \textit{bicyclic or cyclic factors} (BCF) \(\gamma_j(G)/\gamma_{j+1}(G)\), for \(j\ge 3\),
of the lower central series,

\item
parametrized \textit{nilpotency relations} with parameters
\(m\ge 4\) (the index of nilpotency),
\(e\ge 3\) (the CF-invariant), and
\(0\le k\le 1\) (the defect of commutativity),

\begin{equation}
\label{eqn:LowNilpotencyRel}
\begin{aligned}
\sigma_i\ne 1,  \text{ for } i\le m-1, &\quad \sigma_i=1, \text{ for } i\ge m, \\
\tau_j\ne 1, \text{ for } j\le e+k,    &\quad \tau_j=1, \text{ for } j\ge e+k+1,
\end{aligned}
\end{equation}

\noindent
making the structure of the BCF precise,
\(\gamma_j(G)=\langle\sigma_j,\tau_j,\gamma_{j+1}(G)\rangle\)
with bicyclic factor \(\gamma_j(G)/\gamma_{j+1}(G)\simeq C_3\times C_3\), for \(3\le j\le e\), and
\(\gamma_j(G)=\langle\sigma_j,\gamma_{j+1}(G)\rangle\)
with cyclic factor \(\gamma_j(G)/\gamma_{j+1}(G)\simeq C_3\), for \(e+1\le j\le m-1\),

\item
\textit{basic connecting relations} between \(s_i\) and \(\sigma_i\), resp. \(t_i\) and \(\tau_i\),

\begin{equation}
\label{eqn:LowBasicConnectingRel}
\begin{aligned}
\sigma_i=s_{i-2}^3,                 &\quad \tau_i=t_{i-2}^{-3}, \quad \text{ for } i\ge 5, \\
s_j=\sigma_j^{-1}\sigma_{j+1}^{-1}, &\quad t_j=\tau_j\tau_{j+1}, \quad \text{ for } j\ge 5,
\end{aligned}
\end{equation}

\item
parametrized \textit{first central relations} with parameters \(-1\le\beta,\delta,\rho\le 1\),

\begin{equation}
\label{eqn:LowCentralRel}
\begin{aligned}
s_2^{(x-1)(y-1)}                          &=
\lbrack s_3,y\rbrack=\lbrack s_3\sigma_3\sigma_4,y\rbrack=\lbrack t_3,x\rbrack=\lbrack t_3\tau_4^{-1}\tau_5^{-1},x\rbrack=\sigma_{m-1}^{-\rho\delta}\in\zeta_1(G), \\
s_2^{\mathrm{tr}_3(x)+\mathrm{tr}_3(y)-3} &=
s_2^3(s_3t_3)^3s_4t_4=\lbrack s_3\sigma_3\sigma_4,x\rbrack=\lbrack t_3\tau_4^{-1}\tau_5^{-1},y\rbrack=\sigma_{m-1}^{\rho\beta}\in\zeta_1(G), \\
\tau_{e+1}                                &= \sigma_{m-1}^{-\rho}\in\zeta_1(G),
\end{aligned}
\end{equation}

\noindent
where \(\mathrm{tr}_3(x)+\mathrm{tr}_3(y)-3=x^2+x+1+y^2+y+1-3=(x-1)^2+3(x-1)+3+3(y-1)+(y-1)^2\) denotes the \textit{trace element},

\item
and parametrized \textit{second central relations} with additional parameters \(-1\le\alpha,\gamma\le 1\),
which can be viewed as \textit{supplementary connecting relations},

\begin{equation}
\label{eqn:LowSupplConnectingRel}
\begin{aligned}
s_4\sigma_4\sigma_5 &= t_4\tau_4^{-1}\tau_5^{-1} = s_2^{-3}\sigma_4\tau_4^{-1} = s_2^{\mathrm{tr}_3(x)+\mathrm{tr}_3(y)-3}=\sigma_{m-1}^{\rho\beta}\in\zeta_1(G), \\
s_3\sigma_3\sigma_4 &= \sigma_{m-2}^{\rho\beta}\sigma_{m-1}^\gamma\tau_e^\delta, \quad
t_3^{-1}\tau_3\tau_4 = \sigma_{m-2}^{\rho\delta}\sigma_{m-1}^\alpha\tau_e^\beta\in\zeta_2(G),
\end{aligned}
\end{equation}

\noindent
where the first centre is bicyclic \(\zeta_1(G)=\langle\sigma_{m-1},\tau_e\rangle\) for \(k=0\),
but cyclic \(\zeta_1(G)=\langle\sigma_{m-1}\rangle\) for \(k=1\),
and the second centre is tricyclic \(\zeta_2(G)=\langle\sigma_{m-2},\sigma_{m-1},\tau_e\rangle\) for \(k=1\) (with \(m\ge 5\)).
The defect \(k\) is related to \(\rho\) by
\(k=0\) \(\Longleftrightarrow\) \(\rho=0\), and
\(k=1\) \(\Longleftrightarrow\) \(\rho=\pm 1\).

\end{itemize}



\subsection{\(p\)-groups of maximal class}
\label{ss:MaxClassP}

\noindent
Most of the statements concerning \(3\)-groups of maximal class in \S\
\ref{ss:MaxClass3}
remain true for \(p\)-groups of maximal class, 
for an arbitrary prime number \(p\ge 2\).

As before, the \textit{two-step centralizer}
\(\chi_2(G):=\lbrace g\in G\mid\lbrack g,u\rbrack\in\gamma_4(G)\text{ for all }u\in\gamma_2(G)\rbrace\)
of the two-step factor group \(\gamma_2(G)/\gamma_4(G)\) of the lower central series 
is the biggest subgroup of \(G\) such that
\(\lbrack\chi_2(G),\gamma_2(G)\rbrack\le\gamma_4(G)\).
It is characteristic, contains the commutator subgroup \(G^\prime\), and
coincides with \(G\) if and only if \(m=3\).
However, we now have to pay more attention to the \textit{defect of commutativity} \(k=k(G)\) of \(G\)
which is defined by

\begin{equation}
\label{eqn:DefectMaxP}
\lbrack\chi_2(G),\gamma_2(G)\rbrack=\gamma_{m-k}(G)=\zeta_k(G),
\end{equation}

\noindent
and displays a broader variety of possible values than for \(p=3\), namely
\(k=0\) for \(3\le m\le 4\), \(0\le k\le m-4\) for \(m\ge 5\),
and \(0\le k\le\min(m-4,p-2)\) for \(m\ge p+1\),
according to Miech
\cite[p. 331]{Mi1}.

Even more care is required for the parametrized presentation of \(G\),
since it will be crucial for determining the annihilator \(\mathfrak{A}\) of \(G\).
In the isomorphism class of a metabelian \(p\)-group \(G\) of maximal class,
and of order \(\lvert G\rvert=p^m\),
there exists a representative  \(G_a^{m}(z,w)\)
whose normalized generators satisfy the following relations
with a fixed system of parameters
\(a=(a(m-k),\ldots,a(m-1))\), \(w\), and \(z\),
according to Miech
\cite[p. 332]{Mi1}:

\begin{itemize}

\item
parametrized \textit{relations for} \(p\)th \textit{powers} of the generators \(x\), \(y\) and of higher commutators \(s_j\),
with parameters \(0\le w,z\le p-1\),

\begin{equation}
\label{eqn:PowerRelMaxP}
\begin{aligned}
x^p                                                     &= s_{m-1}^w\in\zeta_1(G), \\
y^p\prod_{\ell=2}^p\,s_\ell^{\binom{p}{\ell}}           &= s_{m-1}^z\in\zeta_1(G), \\
s_{j+1}^p\prod_{\ell=2}^p\,s_{j+\ell}^{\binom{p}{\ell}} &= 1, \quad \text{ for } 1\le j\le m-2,
\end{aligned}
\end{equation}

\item
 and a parametrized \textit{commutator relation} with parameters \(0\le a(m-\ell)\le p-1\) for \(1\le\ell\le k\),
 in particular, with non-vanishing parameter \(a(m-k)>0\),

\begin{equation}
\label{eqn:CommutatorRelMaxP}
\lbrack s_2,y\rbrack=\prod_{\ell=1}^k s_{m-\ell}^{a(m-\ell)}\in\lbrack\gamma_2(G),\chi_2(G)\rbrack=\gamma_{m-k}(G).
\end{equation}

\end{itemize}



\section{Annihilators of metabelian \(3\)-groups of maximal class}
\label{s:MaxSmbOrd3}

\begin{remark}
\label{rmk:MaxSmbOrd3}
We continue the traditional notation for particular ideals
occurring as annihilators \(\mathfrak{A}=\mathfrak{A}(G)\) of finite two-generated metabelian \(p\)-groups \(G\),
which was initiated by Scholz and Taussky
\cite{SoTa}
by using old German fraktur letters
\(\mathfrak{L}\), \(\mathfrak{X}\), \(\mathfrak{Z}\), \(\mathfrak{R}\), \(\mathfrak{T}\), \(\mathfrak{V}\).
In the same manner we are going to define new ideals
\(\mathfrak{W}\), \(\mathfrak{Y}\), \(\mathfrak{S}\), \(\mathfrak{U}\),
which were not considered in
\cite{SoTa}
yet. 
\end{remark}


\begin{theorem}
\label{thm:SmbOrdMax3}

(Main Theorem for \(r=1\).)
For a metabelian \(3\)-group \(G\) of coclass \(r=\mathrm{cc}(G)=1\), nilpotency class \(c=\mathrm{cl}(G)=m-1\ge 2\),
that is, with nilpotency index \(m\ge 3\), and defect \(0\le k\le 1\),
which is isomorphic to the representative \(G_\gamma^m(\beta,\alpha)\) with parameters \(-1\le\alpha,\beta,\gamma\le 1\),
where \(\gamma=0\) for \(m\le 4\), the annihilator ideal \(\mathfrak{A}=\mathfrak{A}(G)\) of \(G\) is given by

\begin{equation}
\label{eqn:SmbOrdMax3WithDef}
\mathfrak{A}=\mathfrak{W}_{m-2}(\gamma):=\left(X^{m-2},\ Y-\gamma X^{m-3},\ T_3(X)\right)
\end{equation}

\noindent
with \(T_3(X):=X^2+3X+3\). In particular, if the defect of \(G\) is \(k=0\), that is, \(\gamma=0\), then

\begin{equation}
\label{eqn:SmbOrdMax3}
\mathfrak{A}=\mathfrak{Y}_{m-2}:=\mathfrak{W}_{m-2}(0)=\left(X^{m-2},\ Y,\ T_3(X)\right).
\end{equation}

\end{theorem}

\begin{proof}
We systematically transform relators of the group \(G\)
into generators of the ideal \(\mathfrak{A}\) in the ring \(\mathbb{Z}\lbrack X,Y\rbrack\).

According to the nilpotency relations in Formula
\eqref{eqn:NilpotencyRelMax3},
we have
\(1=s_m=s_2^{(x-1)^{m-2}}\) for \(j=m\),
but
\(1\ne s_{m-1}=s_2^{(x-1)^{m-3}}\) for \(j=m-1\).
Thus, the definition of the epimorphism \(\Psi\) with kernel \(\ker(\Psi)=\mathfrak{A}\) shows that
\(X^{m-2}\in\mathfrak{A}\), but \(X^{m-3}\notin\mathfrak{A}\).

By the commutator relations in Formula
\eqref{eqn:CommutatorRelMax3},
we obtain
\(1=\lbrack s_3,y\rbrack=s_3^{y-1}=s_2^{(x-1)(y-1)}\) for \(j=3\),
and \(s_2^{y-1}=\lbrack s_2,y\rbrack=s_{m-1}^\gamma=s_2^{(x-1)^{m-3}\cdot\gamma}\),
resp. \(1=s_2^{(y-1)-\gamma(x-1)^{m-3}}\) for \(j=2\),
which implies the inclusions
\(XY\in\mathfrak{A}\) and \(Y-\gamma X^{m-3}\in\mathfrak{A}\).
However, the monomial \(XY\) is superfluous as a generator of \(\mathfrak{A}\),
since it is a multiple
\(X\cdot\left(Y-\gamma X^{m-3}\right)=XY-\gamma X^{m-2}\equiv XY\pmod{\mathfrak{A}}\)
of \(Y-\gamma X^{m-3}\).

The third power relation in Formula
\eqref{eqn:PowerRelMax3}
yields
\(1=s_2^3s_3^3s_4=s_2^3s_2^{(x-1)\cdot 3}s_2^{(x-1)^2}=s_2^{(x-1)^2+3(x-1)+3}\)
\(=s_2^{x^2+x+1}=s_2^{\mathrm{tr}_3(x)}\) for \(j=2\).
Consequently, the pre-image \(X^2+3X+3\in\mathfrak{A}\) of the \textit{trace element}
\(\mathrm{tr}_3(x):=x^2+x+1=(x-1)^2+3(x-1)+3\)
under \(\psi\) lies in the annihilator.
\end{proof}


\noindent
The groups in Theorem
\ref{thm:SmbOrdMax3},
which can be realized as second \(3\)-class groups of quadratic fields
\cite{Ma4},
have transfer kernel types (TKTs) \(\mathrm{a.1/2/3}\)
\cite{Ma2}.


In the preceding proof, we have not exploited all power relations yet.
The remaining relations for the powers \(x^3,y^3\) of the normalized generators \(x,y\)
can be used for determining two particular polynomials \(F_1,F_2\in\mathbb{Z}\lbrack X,Y\rbrack\)
which enter the \textit{Schreier conditions} for the existence of a metabelian \(3\)-group in
\cite[\S\ 3, Satz III]{Sr1}
and
\cite[pp. 321--322, and Satz 1, p. 325]{Sr2}.

\begin{corollary}
\label{cor:SchreierRelMax3}

(Schreier conditions.)
Under the assumptions of Theorem
\ref{thm:SmbOrdMax3},
the powers \(x^3,y^3\) of the normalized generators \(x,y\) of \(G\)
can be expressed as \(x^3=s_2^{F_1}\) and \(y^3=s_2^{F_2}\)
with polynomials

\begin{equation}
\label{eqn:SchreierRelMax3}
\begin{aligned}
 & F_1=\alpha X^{m-3}, \quad F_2=\beta X^{m-3}-X-3 \quad \in\mathbb{Z}\lbrack X,Y\rbrack, \quad
 \text{satisfying the Schreier conditions} \\
 & F_1\cdot X\equiv 0, \quad F_2\cdot Y\equiv 0, \quad F_1\cdot Y\equiv -(X^2+3X+3), \quad F_2\cdot X\equiv Y^2+3Y+3 \pmod{\mathfrak{A}}.
\end{aligned}
\end{equation}

\end{corollary}

\begin{proof}
Since the abelianization \(G/G^\prime\) of the group \(G=\langle x,y\rangle\) is of type \((3,3)\),
the powers \(x^3,y^3\) of the normalized generators \(x,y\) are contained in the commutator subgroup \(G^\prime\).
According to Furtw\"angler's epimorphism
\(\Psi:\,\mathbb{Z}\lbrack X,Y\rbrack\to G^\prime\), \(f(X,Y)\mapsto s_2^{f(x-1,y-1)}\),
in Formula
\eqref{eqn:Furtwaengler},
there exist \(F_1,F_2\in\mathbb{Z}\lbrack X,Y\rbrack\) such that
\(x^3=s_2^{F_1(x-1,y-1)}\) and \(y^3=s_2^{F_2(x-1,y-1)}\).
Using the first and second power relation in Formula
\eqref{eqn:PowerRelMax3}
we therefore obtain

\begin{equation*}
\begin{aligned}
s_2^{F_1(x-1,y-1)}=x^3                       &= s_{m-1}^\alpha=s_2^{(x-1)^{m-3}\cdot\alpha}, \\
s_2^{F_2(x-1,y-1)}s_2^3s_2^{x-1}=y^3s_2^3s_3 &= s_{m-1}^\beta=s_2^{(x-1)^{m-3}\cdot\beta},
\end{aligned}
\end{equation*}

\noindent
resp. \(s_2^{F_1(x-1,y-1)-\alpha(x-1)^{m-3}}=1\), \(s_2^{F_2(x-1,y-1)+3+(x-1)-\beta(x-1)^{m-3}}=1\),
and thus the polynomials
\(F_1(X)\equiv\alpha X^{m-3}\), \(F_2(X)\equiv\beta X^{m-3}-X-3\) modulo \(\mathfrak{A}\)
are in fact \textit{univariate}.

In the present context, we have the commutator relations
\(\lbrack x^3,x\rbrack=1\),
\(\lbrack y^3,y\rbrack=1\), trivially, and
\(\lbrack x^3,y\rbrack=\lbrack x,y\rbrack^{\mathrm{tr}_3(x)}\) with \(\mathrm{tr}_3(x)=1+x+x^2=(x-1)^2+3(x-1)+3\),\\
\(\lbrack y^3,x\rbrack=\lbrack y,x\rbrack^{\mathrm{tr}_3(y)}\) with \(\mathrm{tr}_3(y)=1+y+y^2=(y-1)^2+3(y-1)+3\),\\
by the general power rule for commutators.

Consequently, the \textit{Schreier conditions} can be expressed in the form\\
\(s_2^{F_1(x-1,y-1)\cdot(x-1)}=x^{3(x-1)}=\lbrack x^3,x\rbrack=1\), \quad
\(s_2^{F_2(x-1,y-1)\cdot(y-1)}=y^{3(y-1)}=\lbrack y^3,y\rbrack=1\), and\\
\(s_2^{F_1(x-1,y-1)\cdot(y-1)}=x^{3(y-1)}=\lbrack x^3,y\rbrack=s_2^{-\mathrm{tr}_3(x)}\), \quad
\(s_2^{F_2(x-1,y-1)\cdot(x-1)}=y^{3(x-1)}=\lbrack y^3,x\rbrack=s_2^{\mathrm{tr}_3(y)}\),
that is, as congruences modulo \(\mathfrak{A}\),
\(F_1(X,Y)\cdot X\equiv F_2(X,Y)\cdot Y\equiv 0\), and with \textit{trace polynomials}
\(F_1(X,Y)\cdot Y\equiv -T_3(X):=-(X^2+3X+3)\), \quad
\(F_2(X,Y)\cdot X\equiv T_3(Y):=Y^2+3Y+3\).
\end{proof}

The \textit{Schreier conditions} were used essentially
by Scholz and Taussky
\cite[(1)--(5), p. 32]{SoTa},
and by Szekeres
\cite[(3.17)--(3.22), pp. 280--281]{Sz1}
and
\cite[(5.12)--(5.15), pp. 344--345]{Sz2}.
They were also mentioned
by Brink
\cite[p. 32]{Br},
and by Nebelung
\cite[p. 14, and Satz 2.4.5 (iv), p. 45]{Ne1}.



\section{Annihilators of metabelian \(3\)-groups of non-maximal class}
\label{s:LowSmbOrd}

\noindent
For the annihilator \(\mathfrak{A}\) of groups \(G\) of coclass \(\mathrm{cc}(G)\ge 2\),
the proof of the analogue of Theorem
\ref{thm:SmbOrdMax3}
is less straight forward,
requires intricate distinctions of cases,
and will be done in several steps,
since groups with small nilpotency class and small coclass need
additional considerations.



\subsection{Succinct survey of exceptional groups}
\label{ss:Exceptions}

\noindent
In this preparatory section,
we need a minimum of familiarity with rooted descendant trees of finite \(3\)-groups,
as presented in detail in
\cite{Ma6}.
Further, we use identifiers of the form \(\langle\text{order},\text{counter}\rangle\)
as given in the SmallGroups database
\cite{BEO1,BEO2}.
Since we are dealing with groups \(G\) having abelianization \(G/G^\prime\) of type \((3,3)\),
the descendant tree \(\mathcal{T}(R)\) of the abelian root \(R=\langle 9,2\rangle\simeq C_3\times C_3\)
contains all the groups under investigation,
when we restrict it to its metabelian skeleton.
The groups \(G\) of maximal class (coclass \(\mathrm{cc}(G)=1\)) in \S\ 
\ref{ss:MaxClass3}
and \S\
\ref{s:MaxSmbOrd3}
arise as successive descendants of step size \(1\) starting from the root \(R\),
as drawn in
\cite[Fig. 3, p. 167]{Ma6}.
Consequently,
they are CF-groups with cyclic factors \(\gamma_j(G)/\gamma_{j+1}(G)\simeq C_3\), for all \(2\le j\le m-1\).

The groups \(G\) of second maximal and lower class (coclass \(\mathrm{cc}(G)\ge 2\)) in \S\
\ref{ss:NonMaxClass3}
and in the present section \S\
\ref{s:LowSmbOrd}
start with seven descendants \(\langle 243,3\ldots 9\rangle\)
of step size \(2\) of the parent \(\langle 27,3\rangle\), which is not coclass settled with nuclear rank \(2\),
as shown in
\cite[Fig. 4, p. 171]{Ma6}.
These seven isomorphism classes of coclass \(r=2\) share the common nilpotency class \(c=3\), resp. nilpotency index \(m=4\),
and belong to the stem \(\Phi_6(0)\) of Hall's isoclinism family \(\Phi_6\)
\cite{Hl}.
Required information on these groups and relational parameters will be given in the Tables
\ref{tbl:SmlBicExpX}
and
\ref{tbl:SmlBicExpY}.
The bifurcation at the vertex \(B=\langle 27,3\rangle\) is the reason why all groups \(G\) of non-maximal class,
which share the common class-\(2\) quotient \(G/\gamma_3(G)\simeq B\),
are non-CF groups (i.e., BCF-groups) whose lower central series contains at least one bicyclic factor
\(\gamma_3(G)/\gamma_4(G)\simeq C_3\times C_3\).

The possibilities for groups with coclass \(r=2\), class \(c=4\), resp. \(m=5\), and defect \(k=1\)
are more extensive and consist of twelve isomorphism classes, namely
\(\langle 729,34\ldots 36\rangle\) in \(\Phi_{40}(0)\),
\(\langle 729,37\ldots 39\rangle\) in \(\Phi_{41}(0)\),
\(\langle 729,44\ldots 47\rangle\) in \(\Phi_{42}(0)\), and
\(\langle 729,56\ldots 57\rangle\) in \(\Phi_{43}(0)\)
\cite{Ef,Jm}.
They are represented by tiny full discs in
\cite[Fig. 4, p. 171]{Ma6}.
Data for them will be given in Table
\ref{tbl:SmlCycExpX}.



\subsection{The smallest power of \(X\in\psi^{-1}\lbrace x-1\rbrace\) within \(\mathfrak{A}\)}
\label{ss:LowPowX}

\noindent
Due to the \textit{nilpotency relations} for \(\sigma_i\) in Formula
\eqref{eqn:LowNilpotencyRel},
\(\sigma_i=1\), for \(i\ge m\),
and the \textit{basic connecting relations} for \(s_i\) in Formula
\eqref{eqn:LowBasicConnectingRel},
\(s_i\sigma_i\sigma_{i+1}=1\), for \(i\ge 5\),
we have an inclusion for the exponent \(m-2\),
\[s_2^{(x-1)^{m-2}}=s_m=\sigma_m^{-1}\sigma_{m+1}^{-1}=1\cdot 1=1,
\text{ and thus } X^{m-2}\in\mathfrak{A}, \text{ for } m\ge 5.\]
For the exponent \(m-3\), however, we obtain an exclusion,
\[s_2^{(x-1)^{m-3}}=s_{m-1}=\sigma_{m-1}^{-1}\sigma_m^{-1}=\sigma_{m-1}^{-1}\cdot 1=\sigma_{m-1}^{-1}\ne 1,
\text{ and thus } X^{m-3}\not\in\mathfrak{A}, \text{ for } m\ge 6.\]

\begin{remark}
\label{rmk:LowPowX}
An alternative argument for the inclusion \(X^{m-2}\in\mathfrak{A}\) is the fact that
the iterated commutator \(s_m\) lies in \(\gamma_m(G)=1\), even for any \(m\ge 4\),
but this kind of argument is not suitable for proving the exclusion \(X^{m-3}\not\in\mathfrak{A}\),
because although \(s_{m-1}\) is contained in \(\gamma_{m-1}(G)>1\), it is nevertheless possible that \(s_{m-1}=1\).
\end{remark}

\begin{proposition}
\label{prp:LowPowX}

Let \(G\) be a metabelian \(3\)-group with abelianization \(G/G^\prime\) of type \((3,3)\).
Suppose \(G\) is of order \(\lvert G\rvert=3^n\ge 3^5\), nilpotency class \(c=\mathrm{cl}(G)=m-1\ge 3\), coclass \(r=\mathrm{cc}(G)=e-1\ge 2\),
and defect \(0\le k\le 1\), where \(4\le m<n\le 2m-3\) and \(e=n-m+2\ge 3\).
Then the annihilator ideal \(\mathfrak{A}\unlhd\mathbb{Z}\lbrack X,Y\rbrack\) of \(G\)
contains \(X^{m-2}\) as the power of \(X\) with minimal exponent.

The unique exception are the four groups with SmallGroup identifiers \(\langle 729,44\ldots 47\rangle\)
\cite{BEO1,BEO2},
which are of order \(3^6\), nilpotency class \(4\), and coclass \(2\),
that is, \(n=6\), \(m=5\), and \(e=3\).
For these groups, the smallest power of \(X\) contained in the annihilator \(\mathfrak{A}\) is \(X^2\) with exponent \(m-3\).

\end{proposition}

\begin{proof}
It still remains to investigate,

\begin{itemize}
\item
if \(X^2\in\mathfrak{A}\) for \(m=4\),
\item
if \(X^2\not\in\mathfrak{A}\) for \(m=5\),
\item
if \(X\not\in\mathfrak{A}\) for \(m=4\).
\end{itemize}

\noindent
To this end, we use the \textit{second central relation} for \(s_4\) in Formula
\eqref{eqn:LowSupplConnectingRel},
\(s_4\sigma_4\sigma_5=\sigma_{m-1}^{\rho\beta}\),
viewed as a \textit{supplementary connecting relation},
which implies that
\[s_2^{(x-1)^2}=s_4=\sigma_4^{-1}\sigma_5^{-1}\sigma_{m-1}^{\rho\beta}.\]
For \(m=4\), where always \(k=0\), and thus \(\rho=0\),
this means that
\[s_2^{(x-1)^2}=\sigma_4^{-1}\sigma_5^{-1}=1\cdot 1=1, \text{ and thus } X^2\in\mathfrak{A}.\]
For \(m=5\), we obtain an expression which depends on the relational parameters \(\rho\) and \(\beta\),
\[s_2^{(x-1)^2}=\sigma_4^{\rho\beta-1}, \text{ that is, }
s_2^{(x-1)^2}=1,\ X^2\in\mathfrak{A}, \text{ for } \rho\beta=1, \text{ and }
s_2^{(x-1)^2}\ne 1,\ X^2\not\in\mathfrak{A}, \text{ for } \rho\beta\ne 1.\]
For the twelve isomorphism classes of groups with \(m=5\), \(n=6\), \(\rho=\pm 1\),
and generally for \(m=5\), \(\rho=0\), Table
\ref{tbl:SmlCycExpX}
gives the relevant information, according to
\cite[p. 4--7]{Ne2},
and
\cite{BEO2}.

\begin{table}[ht]
\caption{Relational parameters for \(m=5\), \(n=6\), \(e=3\), \(\rho=\pm 1\) or \(\rho=0\)}
\label{tbl:SmlCycExpX}
\begin{center}
\begin{tabular}{|c|l|c|rrr|c|}
\hline
 Groups                             & TKT  & Classes & \(\beta\) & \(\delta\) & \(\rho\) & \(\sigma_4^{\rho\beta-1}\) \\
\hline
 \(\langle 729,37\ldots 39\rangle\) & b.10 &   \(3\) &     \(0\) &      \(0\) &    \(1\) &          \(\sigma_4^{-1}\) \\
 \(\langle 729,44\ldots 47\rangle\) & H.4  &   \(4\) &     \(1\) &      \(1\) &    \(1\) &                      \(1\) \\
 \(\langle 729,56\ldots 57\rangle\) & G.19 &   \(2\) &    \(-1\) &      \(0\) &    \(1\) &          \(\sigma_4^{-2}\) \\
 \(\langle 729,34\ldots 36\rangle\) & b.10 &   \(3\) &     \(0\) &      \(0\) &   \(-1\) &          \(\sigma_4^{-1}\) \\
\hline
 generally                          &      &         &           &            &    \(0\) &          \(\sigma_4^{-1}\) \\
\hline
\end{tabular}
\end{center}
\end{table}

\noindent
Only in the exceptional case of the four isomorphism classes \(\langle 729,44\ldots 47\rangle\)
with TKT (transfer kernel type) H.4
\cite{Ma2},
\(m=5\), and \(\rho\beta=1\), we obtain an inclusion \(X^2\in\mathfrak{A}\).
In the other cases of \(m=5\), we always have \(\rho\beta\ne 1\), and thus \(X^2\not\in\mathfrak{A}\).

\noindent
Finally, we investigate,
if \(X\not\in\mathfrak{A}\) for \(m=4\) and for the TKT H.4 with \(m=5\), \(\rho\beta=1\),
by employing the \textit{second central relation} for \(s_3\) in Formula
\eqref{eqn:LowSupplConnectingRel},
\(s_3\sigma_3\sigma_4=\sigma_{m-2}^{\rho\beta}\sigma_{m-1}^\gamma\tau_e^\delta\),
viewed as a \textit{supplementary connecting relation},
which shows that
\[s_2^{x-1}=s_3=\sigma_3^{-1}\sigma_4^{-1}\sigma_{m-2}^{\rho\beta}\sigma_{m-1}^\gamma\tau_e^\delta,\]
that is, \(s_2^{x-1}=\sigma_3^{\gamma-1}\tau_3^\delta\) for \(m=4\), where \(\rho=0\) and \(e=3\),
and \(s_2^{x-1}=\sigma_3^{\rho\beta-1}\sigma_4^{\gamma-1}\tau_e^\delta\) for \(m=5\).
For the seven isomorphism classes of groups with \(m=4\),
we have Table
\ref{tbl:SmlBicExpX},
by
\cite[p. 1--3]{Ne2},
and
\cite{BEO2}.

\begin{table}[ht]
\caption{Relational parameters for \(m=4\), \(n=5\), \(e=3\)}
\label{tbl:SmlBicExpX}
\begin{center}
\begin{tabular}{|c|l|rrrr|c|}
\hline
 Group                    & TKT  & \(\alpha\) & \(\beta\) & \(\gamma\) & \(\delta\) & \(\sigma_3^{\gamma-1}\tau_3^\delta\) \\
\hline
 \(\langle 243,3\rangle\) & b.10 &      \(0\) &     \(0\) &      \(0\) &      \(0\) &                    \(\sigma_3^{-1}\) \\
 \(\langle 243,8\rangle\) & c.21 &      \(0\) &     \(0\) &      \(0\) &      \(1\) &              \(\sigma_3^{-1}\tau_3\) \\
 \(\langle 243,6\rangle\) & c.18 &      \(0\) &    \(-1\) &      \(0\) &      \(1\) &              \(\sigma_3^{-1}\tau_3\) \\
 \(\langle 243,5\rangle\) & D.10 &      \(0\) &     \(0\) &     \(-1\) &      \(1\) &              \(\sigma_3^{-2}\tau_3\) \\
 \(\langle 243,9\rangle\) & G.19 &      \(0\) &    \(-1\) &     \(-1\) &      \(0\) &                    \(\sigma_3^{-2}\) \\
 \(\langle 243,4\rangle\) & H.4  &      \(1\) &     \(1\) &      \(1\) &      \(1\) &                           \(\tau_3\) \\
 \(\langle 243,7\rangle\) & D.5  &      \(1\) &     \(1\) &     \(-1\) &      \(1\) &              \(\sigma_3^{-2}\tau_3\) \\
\hline
\end{tabular}
\end{center}
\end{table}

\noindent
Consequently, we always have \(s_2^{x-1}\ne 1\) and \(X\not\in\mathfrak{A}\) for \(m=4\).
In the case of TKT H.4 with \(m=5\), \(\rho\beta=1\),
where \(\delta=1\) and \(\gamma=1,-1,-1,0\), depending on the isomorphism class,
we also get \(s_2^{x-1}=\sigma_3^{\gamma-1}\tau_3\ne 1\) and \(X\not\in\mathfrak{A}\).
\end{proof}



\subsection{The smallest power of \(Y\in\psi^{-1}\lbrace y-1\rbrace\) within \(\mathfrak{A}\)}
\label{ss:LowPowY}

To avoid confusion, we successively treat the cases
\(k=0\), that is \(\lbrack\chi_s(G),\gamma_e(G)\rbrack=1\), and
\(k=1\), that is \(\lbrack\chi_s(G),\gamma_e(G)\rbrack=\gamma_{m-1}(G)\).

First, let \(\lbrack\chi_s(G),\gamma_e(G)\rbrack=1\), that is \(k=0\) and \(\rho=0\).\\
In view of the \textit{nilpotency relations} for \(\tau_i\) in Formula
\eqref{eqn:LowNilpotencyRel},
\(\tau_i=1\), for \(i\ge e+1\),
and the \textit{basic connecting relations} for \(t_i\) in Formula
\eqref{eqn:LowBasicConnectingRel},
\(t_i\tau_i^{-1}\tau_{i+1}^{-1}=1\), for \(i\ge 5\),
we have an inclusion for the exponent \(e-1\),
\[t_2^{(y-1)^{e-1}}=t_{e+1}=\tau_{e+1}\tau_{e+2}=1\cdot 1=1, \text{ and thus } Y^{e-1}\in\mathfrak{A} \text{ for } e\ge 4.\]
For the exponent \(e-2\), however, we obtain the exclusion
\[t_2^{(y-1)^{e-2}}=t_e=\tau_e\tau_{e+1}=\tau_e\cdot 1=\tau_e\ne 1, \text{ and thus } Y^{e-2}\not\in\mathfrak{A} \text{ for } e\ge 5.\]

\begin{remark}
\label{rmk:LowPowY}
Here, the alternative argument for the inclusion \(Y^{e-1}\in\mathfrak{A}\),
that \(t_{e+1}\) is an iterated commutator in \(\gamma_{e+1}(G)\),
can only be applied for the maximal value \(e=m-1\), but not for \(e<m-1\).
\end{remark}

\begin{proposition}
\label{prp:LowPowY}

Let \(G\) be a metabelian \(3\)-group with abelianization \(G/G^\prime\) of type \((3,3)\).
Suppose \(G\) is of order \(\lvert G\rvert=3^n\ge 3^5\), nilpotency class \(c=\mathrm{cl}(G)=m-1\ge 3\), coclass \(r=\mathrm{cc}(G)=e-1\ge 2\),
and defect \(k=0\), where \(4\le m<n\le 2m-3\) and \(e=n-m+2\ge 3\).
Then the annihilator ideal \(\mathfrak{A}\unlhd\mathbb{Z}\lbrack X,Y\rbrack\) of \(G\)
contains \(Y^{e-1}\) as the power of \(Y\) with minimal exponent.

\end{proposition}

\begin{proof}
It remains to investigate,

\begin{itemize}
\item
if \(Y^2\in\mathfrak{A}\) for \(e=3\),
\item
if \(Y^2\not\in\mathfrak{A}\) for \(e=4\),
\item
if \(Y\not\in\mathfrak{A}\) for \(e=3\).
\end{itemize}

\noindent
We employ the \textit{second central relation} for \(t_4\) in Formula
\eqref{eqn:LowSupplConnectingRel},
\(t_4\tau_4^{-1}\tau_5^{-1}=\sigma_{m-1}^{\rho\beta}=1\),
viewed as a \textit{supplementary connecting relation},
which implies that
\[t_2^{(y-1)^2}=t_4=\tau_4\tau_5.\]
For \(e=3\), this means
\(t_2^{(y-1)^2}=1\cdot 1=1\), and thus \(Y^2\in\mathfrak{A}\).\\
For \(e=4\), we obtain
\(t_2^{(y-1)^2}=\tau_4\cdot 1=\tau_4\ne 1\), and thus \(Y^2\not\in\mathfrak{A}\).

\noindent
Finally, we investigate
if \(Y\not\in\mathfrak{A}\) for \(e=3\),
by using the \textit{second central relation} for \(t_3\) in Formula
\eqref{eqn:LowSupplConnectingRel},
\(t_3\tau_3^{-1}\tau_4^{-1}=\sigma_{m-2}^{-\rho\delta}\sigma_{m-1}^{-\alpha}\tau_e^{-\beta}=\sigma_{m-1}^{-\alpha}\tau_e^{-\beta}\),
viewed as a \textit{supplementary connecting relation},
which yields
\[t_2^{y-1}=t_3=\tau_3\tau_4\tau_e^{-\beta}\sigma_{m-1}^{-\alpha},\]
that is \(t_2^{y-1}=\tau_3^{1-\beta}\sigma_{m-1}^{-\alpha}\) for \(e=3\), where \(\tau_4=1\).
Therefore, we have to consider the isomorphism classes of groups with \(e=3\), \(\rho=0\),
firstly for \(m=4\) in Table
\ref{tbl:SmlBicExpY},
using
\cite[p. 1--3]{Ne2},
and
\cite{BEO2},
then for \(m\ge 5\) odd, and eventually for \(m\ge 6\) even (with the number of classes and signs in parentheses),
which is done in Table
\ref{tbl:SecExp},
using
\cite[pp. 8--12]{Ne2}
for infinitely many groups.

\begin{table}[ht]
\caption{Relational parameters for \(m=4\), \(n=5\), \(e=3\)}
\label{tbl:SmlBicExpY}
\begin{center}
\begin{tabular}{|c|l|rrrr|c|}
\hline
 Group                    & TKT   & \(\alpha\) & \(\beta\) & \(\gamma\) & \(\delta\) & \(\sigma_3^{-\alpha}\tau_3^{1-\beta}\) \\
\hline
 \(\langle 243,3\rangle\) &  b.10 &      \(0\) &     \(0\) &      \(0\) &      \(0\) &                             \(\tau_3\) \\
 \(\langle 243,8\rangle\) &  c.21 &      \(0\) &     \(0\) &      \(0\) &      \(1\) &                             \(\tau_3\) \\
 \(\langle 243,6\rangle\) &  c.18 &      \(0\) &    \(-1\) &      \(0\) &      \(1\) &                           \(\tau_3^2\) \\
 \(\langle 243,5\rangle\) &  D.10 &      \(0\) &     \(0\) &     \(-1\) &      \(1\) &                             \(\tau_3\) \\
 \(\langle 243,9\rangle\) &  G.19 &      \(0\) &    \(-1\) &     \(-1\) &      \(0\) &                           \(\tau_3^2\) \\
 \(\langle 243,4\rangle\) &  H.4  &      \(1\) &     \(1\) &      \(1\) &      \(1\) &                      \(\sigma_3^{-1}\) \\
 \(\langle 243,7\rangle\) &  D.5  &      \(1\) &     \(1\) &     \(-1\) &      \(1\) &                      \(\sigma_3^{-1}\) \\
\hline
\end{tabular}
\end{center}
\end{table}

\smallskip
\noindent
Thus, we always have \(t_2^{y-1}\ne 1\) and \(Y\notin\mathfrak{A}\)
for the \(7\) isomorphism classes with \(e=3\), \(m=4\).

\smallskip
\begin{table}[h]
\caption{Relational parameters for \(m\ge 5\), \(n\ge 6\), \(e=3\)}
\label{tbl:SecExp}
\begin{center}
\begin{tabular}{|l|c|rrrr|c|}
\hline
 Type &   Classes & \(\alpha\) & \(\beta\) & \(\gamma\) & \(\delta\) & \(\sigma_{m-1}^{-\alpha}\tau_3^{1-\beta}\) \\
\hline
 b.10 &     1 (1) &      \(0\) &     \(0\) &      \(0\) &      \(0\) &                                 \(\tau_3\) \\
 d.25 &     1 (2) &      \(0\) &     \(0\) & \((\pm)1\) &      \(0\) &                                 \(\tau_3\) \\
 d.23 &     1 (1) &      \(1\) &     \(0\) &      \(0\) &      \(0\) &                \(\sigma_{m-1}^{-1}\tau_3\) \\
 d.19 &     1 (2) &      \(1\) &     \(0\) & \((\pm)1\) &      \(0\) &                \(\sigma_{m-1}^{-1}\tau_3\) \\
\hline
 c.21 &     1 (1) &      \(0\) &     \(0\) &      \(0\) &      \(1\) &                                 \(\tau_3\) \\
 E.9  &     1 (2) &      \(0\) &     \(0\) & \((\pm)1\) &      \(1\) &                                 \(\tau_3\) \\
 G.16 &     1 (2) & \((\pm)1\) &     \(0\) &      \(0\) &      \(1\) &            \(\sigma_{m-1}^{(\mp)1}\tau_3\) \\
 E.8  &     1 (1) &      \(1\) &     \(0\) &     \(-1\) &      \(1\) &                \(\sigma_{m-1}^{-1}\tau_3\) \\
\hline
 c.18 &     1 (1) &      \(0\) &    \(-1\) &      \(0\) &      \(1\) &                               \(\tau_3^2\) \\
 E.14 &     1 (2) &      \(0\) &    \(-1\) & \((\pm)1\) &      \(1\) &                               \(\tau_3^2\) \\
 E.6  &     1 (1) &      \(1\) &    \(-1\) &      \(1\) &      \(1\) &              \(\sigma_{m-1}^{-1}\tau_3^2\) \\
 H.4  &     1 (2) & \((\pm)1\) &    \(-1\) & \((\mp)1\) &      \(1\) &          \(\sigma_{m-1}^{(\mp)1}\tau_3^2\) \\
\hline
\end{tabular}
\end{center}
\end{table}

\noindent
As before, there is an exclusion \(t_2^{y-1}\ne 1\) and \(Y\notin\mathfrak{A}\) for both,
the \(12\) isomorphism classes with \(e=3\), for each odd \(m\ge 5\),
and the \(18\) isomorphism classes with \(e=3\), for each even \(m\ge 6\).
\end{proof}


Now, let \(\lbrack\chi_s(G),\gamma_e(G)\rbrack=\gamma_{m-1}(G)\), that is \(k=1\) and \(\rho=\pm 1\).\\
In view of the \textit{nilpotency relations} for \(\tau_i\) in Formula
\eqref{eqn:LowNilpotencyRel},
\(\tau_i=1\), for \(i\ge e+2\),
and the \textit{basic connecting relations} for \(t_i\) in Formula
\eqref{eqn:LowBasicConnectingRel},
\(t_i\tau_i^{-1}\tau_{i+1}^{-1}=1\), for \(i\ge 5\),
we have an inclusion for the exponent \(e\),
\[t_2^{(y-1)^{e}}=t_{e+2}=\tau_{e+2}\tau_{e+3}=1\cdot 1=1, \text{ and thus } Y^{e}\in\mathfrak{A} \text{ for } e\ge 3.\]
For the exponent \(e-1\), however, we obtain the exclusion
\[t_2^{(y-1)^{e-1}}=t_{e+1}=\tau_{e+1}\tau_{e+2}=\tau_{e+1}\cdot 1=\tau_{e+1}\ne 1, \text{ and thus } Y^{e-1}\not\in\mathfrak{A} \text{ for } e\ge 4.\]

\begin{proposition}
\label{prp:LowPowYWithDefect}

Let \(G\) be a metabelian \(3\)-group with abelianization \(G/G^\prime\) of type \((3,3)\).
Suppose \(G\) is of order \(\lvert G\rvert=3^n\ge 3^5\), nilpotency class \(c=\mathrm{cl}(G)=m-1\ge 3\), coclass \(r=\mathrm{cc}(G)=e-1\ge 2\),
and defect \(k=1\), where \(5\le m<n<2m-3\) and \(e=n-m+2\ge 3\).
Then the annihilator ideal \(\mathfrak{A}\unlhd\mathbb{Z}\lbrack X,Y\rbrack\) of \(G\)
contains \(Y^{e}\) as the power of \(Y\) with minimal exponent.

The unique exception are the four groups with SmallGroup identifiers \(\langle 729,44\ldots 47\rangle\)
\cite{BEO1,BEO2},
which are of order \(3^6\), nilpotency class \(4\), and coclass \(2\),
that is, \(n=6\), \(m=5\), and \(e=3\).
For these groups, the smallest power of \(Y\) contained in the annihilator \(\mathfrak{A}\) is \(Y^2\) with exponent \(e-1\).
\end{proposition}

\begin{proof}
It remains to investigate,

\begin{itemize}
\item
if \(Y^2\not\in\mathfrak{A}\) for \(e=3\),
\item
if \(Y\not\in\mathfrak{A}\) for \(e=3\).
\end{itemize}

\noindent
Again, we employ the \textit{second central relation} for \(t_4\) in Formula
\eqref{eqn:LowSupplConnectingRel},
\(t_4\tau_4^{-1}\tau_5^{-1}=\sigma_{m-1}^{\rho\beta}\),
viewed as a \textit{supplementary connecting relation},
which implies that
\[t_2^{(y-1)^2}=t_4=\tau_4\tau_5\sigma_{m-1}^{\rho\beta}.\]
For \(e=3\), where \(\tau_5=1\), this means
\(t_2^{(y-1)^2}=\tau_4\cdot 1\cdot\sigma_{m-1}^{\rho\beta}=\tau_{e+1}\sigma_{m-1}^{\rho\beta}\).
At this point, we need the \textit{first central relation} for \(\tau_{e+1}\) in Formula
\eqref{eqn:LowCentralRel},
\(\tau_{e+1}=\sigma_{m-1}^{-\rho}\),
which yields

\begin{equation}
\label{eqn:SecCentralRel}
t_2^{(y-1)^2}=\tau_{e+1}\tau_{e+1}^{-\beta}=\tau_{4}^{1-\beta}=\sigma_{m-1}^{(\beta-1)\rho}.
\end{equation}

\noindent
Consequently, we must investigate, for which of the infinitely many isomorphism classes
with \(e=3\) and \(\rho=\pm 1\) the parameter \(\beta\) takes the value \(\beta=1\). 
According to
\cite[pp. 13--33]{Ne2},
this exclusively happens for the same four isomorphism classes with TKT H.4 and \(m=5\)
which gave rise to the exception in Proposition
\ref{prp:LowPowX}
already, but it never occurs for \(m\ge 6\).
For these exceptions, we have \(Y^2\in\mathfrak{A}\), otherwise always \(Y^2\notin\mathfrak{A}\).

Finally, we check
if \(Y\not\in\mathfrak{A}\) for the groups \(\langle 729,44\ldots 47\rangle\)
with \(e=3\), \(m=5\), and \(\beta=1\),
by using the \textit{second central relation} for \(t_3\) in Formula
\eqref{eqn:LowSupplConnectingRel},
\(t_3\tau_3^{-1}\tau_4^{-1}=\sigma_{m-2}^{-\rho\delta}\sigma_{m-1}^{-\alpha}\tau_e^{-\beta}=\sigma_{m-1}^{-\alpha}\tau_e^{-\beta}\),
viewed as a \textit{supplementary connecting relation},
which yields
\[t_2^{y-1}=t_3=\tau_3\tau_4\tau_e^{-\beta}\sigma_{m-2}^{-\rho\delta}\sigma_{m-1}^{-\alpha},\]
that is, \(t_2^{y-1}=\tau_3^{1-\beta}\tau_4\sigma_{3}^{-\rho\delta}\sigma_{4}^{-\alpha}\) for \(e=3\), \(m=5\),
and, together with \(\beta=\delta=\rho=1\) and the above mentioned formula \(\tau_{4}=\tau_{e+1}=\sigma_{m-1}^{-\rho}=\sigma_{4}^{-1}\),
finally \(t_2^{y-1}=\sigma_{3}^{-1}\sigma_{4}^{-1-\alpha}\ne 1\), \(Y\notin\mathfrak{A}\).
\end{proof}



\subsection{Further polynomial generators of \(\mathfrak{A}\)}
\label{ss:LowFurtherPolynomials}

\begin{proposition}
\label{prp:LowFurtherPolynomials}

Let \(G\) be a metabelian \(3\)-group with abelianization \(G/G^\prime\) of type \((3,3)\).
Suppose \(G\) is of order \(\lvert G\rvert=3^n\ge 3^5\), nilpotency class \(c=\mathrm{cl}(G)=m-1\ge 3\), coclass \(r=\mathrm{cc}(G)=e-1\ge 2\),
and defect \(0\le k\le 1\), where \(4\le m<n\le 2m-3\) and \(3\le e=n-m+2\le m-1\).
Denote the \textit{trace polynomial} by \(T:=X^2+3X+3+3Y+Y^2\).

\begin{enumerate}

\item
If \(k=0\), then \(\rho=0\) and
the annihilator ideal \(\mathfrak{A}\unlhd\mathbb{Z}\lbrack X,Y\rbrack\) of \(G\)
contains the bivariate polynomials \(XY\), \quad \(T\), \quad and \(Y^{e-1}\).

\item
If \(k=1\) and \(m\ge 6\),
then \(\mathfrak{A}\) contains
\(XY-\rho\delta X^{m-3}\), \quad \(T+\rho\beta X^{m-3}\), \\
\quad and \(Y^{e-1}-\rho X^{m-3}\), if \(e\ge 4\), \quad but \(Y^2-(1-\beta)\rho X^{m-3}\), if \(e=3\).

\item
If \(k=1\), \(m=5\), and \(\rho\beta\ne 1\),
then \(e=3\) and \(\mathfrak{A}\) contains\\
\(XY-(1-\rho\beta)\rho\delta X^2\), \quad \(T+(1-\rho\beta)\rho\beta X^2\), \quad
and \(Y^2-(1-\beta)(1-\rho\beta)\rho X^2\).

\item
If \(k=1\), \(m=5\), and \(\rho\beta=1\),
then \(e=3\) and \(\mathfrak{A}\) contains
\(XY+3\), \quad \(X^2\), \quad and \(Y^2\).

\end{enumerate}

\end{proposition}

\begin{proof}
The \textit{first central relations} in Formula
\eqref{eqn:LowCentralRel}
provide more intricate bivariate polynomials lying in the annihilator \(\mathfrak{A}\)
than the simple monomials \(X^{m-2}\) and \(Y^{e-1}\), resp. \(Y^e\).

\noindent
(1)
If the defect \(k=0\) vanishes, then the parameter \(\rho=0\) is also equal to zero,
and we immediately obtain,
firstly, \(s_2^{(x-1)(y-1)}=\sigma_{m-1}^{-\rho\delta}=1\), and thus \(XY\in\mathfrak{A}\), secondly,
\(s_2^{\mathrm{tr}_3(x)+\mathrm{tr}_3(y)-3}=\sigma_{m-1}^{\rho\beta}=1\), and thus \(X^2+3X+3+3Y+Y^2\in\mathfrak{A}\).
By Proposition
\ref{prp:LowPowY},
we know that \(Y^{e-1}\in\mathfrak{A}\).

\noindent
(2)
If the defect \(k=1\) is positive, however,
all these relations contain the element \(\sigma_{m-1}\in\gamma_{m-1}(G)\le\zeta_1(G)\),
and the \textit{basic connecting relations} for \(s_i\) in Formula
\eqref{eqn:LowBasicConnectingRel},
\(s_i\sigma_i\sigma_{i+1}=1\), for \(i\ge 5\),
must be used to express \(\sigma_{m-1}\) in terms of \(s_{m-1}\)
by putting \(i=m-1\ge 5\), which is only possible for \(m\ge 6\).
The preparatory result is
\(s_2^{(x-1)^{m-3}}=s_{m-1}=\sigma_{m-1}^{-1}\sigma_m^{-1}=\sigma_{m-1}^{-1}\cdot 1=\sigma_{m-1}^{-1}\).

Together, we firstly obtain
\(s_2^{(x-1)(y-1)}=\sigma_{m-1}^{-\rho\delta}=s_2^{(x-1)^{m-3}\cdot\rho\delta}\),
resp. \(1=s_2^{(x-1)(y-1)-\rho\delta(x-1)^{m-3}}\),
and thus \(XY-\rho\delta X^{m-3}\in\mathfrak{A}\), if \(m\ge 6\).

Secondly, we have
\(s_2^{\mathrm{tr}_3(x)+\mathrm{tr}_3(y)-3}=\sigma_{m-1}^{\rho\beta}=s_2^{-(x-1)^{m-3}\cdot\rho\beta}\),
resp. \(1=s_2^{\mathrm{tr}_3(x)+\mathrm{tr}_3(y)-3+\rho\beta(x-1)^{m-3}}\),
and thus \(X^2+3X+3+3Y+Y^2+\rho\beta X^{m-3}\in\mathfrak{A}\), if \(m\ge 6\).

Finally, we get
\(\tau_{e+1}=\sigma_{m-1}^{-\rho}=s_2^{(x-1)^{m-3}\cdot\rho}\),
but here we must also use
the \textit{basic connecting relations} for \(t_i\) in Formula
\eqref{eqn:LowBasicConnectingRel},
\(t_i\tau_i^{-1}\tau_{i+1}^{-1}=1\), for \(i\ge 5\).
To express \(\tau_{e+1}\) in terms of \(t_{e+1}\),
we put \(i=e+1\ge 5\), which is only possible for \(e\ge 4\).

This yields
\(t_2^{(y-1)^{e-1}}=t_{e+1}=\tau_{e+1}\tau_{e+2}=\tau_{e+1}\cdot 1=\tau_{e+1}\).
Since \(t_2=s_2\), it follows that
\(1=s_2^{(y-1)^{e-1}-\rho(x-1)^{m-3}}\),
and thus \(Y^{e-1}-\rho X^{m-3}\in\mathfrak{A}\), provided that \(k=1\), \(m\ge 6\), and \(e\ge 4\).

For \(e=3\) we use Formula
\eqref{eqn:SecCentralRel},
\(t_2^{(y-1)^2}=\sigma_{m-1}^{(\beta-1)\rho}=s_2^{(x-1)^{m-3}\cdot (1-\beta)\rho}\),
which implies that \(Y^2-(1-\beta)\rho X^{m-3}\in\mathfrak{A}\), for \(m\ge 6\).

\noindent
(3)
Since \(k=1\) cannot occur for \(m=4\),
it remains to consider the extreme value \(m=5\) of the nilpotency index.
For this purpose, we use the \textit{second central relation} for \(s_4\) in Formula
\eqref{eqn:LowSupplConnectingRel},
\(s_4\sigma_4\sigma_5=\sigma_{m-1}^{\rho\beta}\),
viewed as a \textit{supplementary connecting relation},
which implies that
\[s_2^{(x-1)^2}=s_4=\sigma_4^{-1}\cdot 1\cdot\sigma_{4}^{\rho\beta}=\sigma_4^{\rho\beta-1}.\]
This relation can be inverted to \(\sigma_4=s_4^{\rho\beta-1}=s_2^{(\rho\beta-1)(x-1)^2}\) if \(\rho\beta\ne 1\).
If \(\rho\beta=1\), however, then \(\sigma_4\) cannot be expressed in terms of \(s_4\), since \(s_4=1\).
An alert for the possibility of such a case was issued in Remark
\ref{rmk:LowPowX}.
For the time being, let \(\rho\beta\ne 1\).
We observe that \(k=1\) enforces \(e\le m-2=3\), and thus \(e=3\),
whence we obtain
\[s_2^{(x-1)(y-1)}=\sigma_{4}^{-\rho\delta}=s_2^{(\rho\beta-1)(x-1)^2\cdot (-\rho\delta)},
\text{ and } XY+(\rho\beta-1)\rho\delta X^2\in\mathfrak{A},\]
\[s_2^{\mathrm{tr}_3(x)+\mathrm{tr}_3(y)-3}=\sigma_{4}^{\rho\beta}=s_2^{(\rho\beta-1)(x-1)^2\cdot\rho\beta},
\text{ and } X^2+3X+3+3Y+Y^2-(\rho\beta-1)\rho\beta X^2\in\mathfrak{A},\]
\[t_2^{(y-1)^2}=\tau_{4}^{1-\beta} \text{ with } \tau_{4}=\sigma_{4}^{-\rho}=s_2^{(\rho\beta-1)(x-1)^2\cdot (-\rho)},
\text{ and } Y^2+(1-\beta)(\rho\beta-1)\rho X^2\in\mathfrak{A}.\]

\noindent
(4)
After all, we turn to the case \(m=5\) and \(\rho\beta=1\).
We have seen that it only occurs for the four isomorphism classes \(\langle 729,44\ldots 47\rangle\),
where \(\beta=\delta=\rho=1\).
The results for this case are completely exceptional:
\(1=s_4=s_2^{(x-1)^2}\) and thus \(X^2\in\mathfrak{A}\),
\(\sigma_4=\lbrack\sigma_3,x\rbrack=\lbrack y^3,x\rbrack=\lbrack y,x\rbrack^{tr_3(y)}=s_2^{Y^2+3Y+3}\),
and therefore
\(XY+\rho\delta(Y^2+3Y+3)=XY+Y^2+3Y+3\in\mathfrak{A}\),
the degenerate triviality
\(X^2+3X+3+3Y+Y^2-\rho\beta(Y^2+3Y+3)=X^2+3X+3+3Y+Y^2-Y^2-3Y-3\equiv 3X\in\mathfrak{A}\),
and \(Y^2+(1-\beta)\rho(Y^2+3Y+3)=Y^2\in\mathfrak{A}\), which implies \(3Y\in\mathfrak{A}\).

In all cases of \(k=1\), we know from Proposition
\ref{prp:LowPowYWithDefect}
that \(Y^{e}\in\mathfrak{A}\), for an assigned CF-invariant \(e\ge 3\),
but \(Y^{e}\) is a multiple of \(Y^{e-1}-\rho X^{m-3}\) and thus superfluous as a generator of \(\mathfrak{A}\).
\end{proof}



\subsection{The general template for annihilator ideals}
\label{ss:TemplateSmbOrd}

\begin{theorem}
\label{thm:TemplateSmbOrd}

Let \(G\) be a metabelian \(3\)-group with abelianization \(G/G^\prime\) of type \((3,3)\).
Suppose that \(G\) is of order \(\lvert G\rvert=3^n\ge 3^5\),
nilpotency class \(c=\mathrm{cl}(G)=m-1\ge 3\) and nilpotency index \(m\ge 4\),
coclass \(r=\mathrm{cc}(G)=e-1\ge 2\) and \(\mathrm{CF}\)-invariant \(3\le e=n-m+2\le m-1\),
and defect of commutativity \(0\le k\le 1\).
Assume that \(G\) is isomorphic to the representative \(G_\rho^{m,n}(\alpha,\beta,\gamma,\delta)\)
with parameters \(-1\le\alpha,\beta,\gamma,\delta,\rho\le 1\), \(m+1\le n\le 2m-3\).
Let \(T(X,Y):=X^2+3X+3+3Y+Y^2\) denote the trace polynomial.
Then the annihilator \(\mathfrak{A}=\mathfrak{A}(G)\) of \(G\) is given as follows.

\begin{enumerate}

\item
If \(k=0\), then, in dependence on the class \(c=m-1\) and the coclass \(r=e-1\),

\begin{equation}
\label{eqn:SmbOrdWithoutDefect}
\mathfrak{A}=\mathfrak{R}_{m-2,e-1}:=\left(X^{m-2},\ Y^{e-1},\ XY,\ T(X,Y)\right)
\end{equation}

\noindent
is completely independent of the parameters  \(\alpha,\beta,\gamma,\delta\).

\item
If \(k=1\), \(m\ge 6\), and \(e\ge 4\), that is, if \(G\) is of coclass \(r\ge 3\) and class \(c\ge 5\), then \(\mathfrak{A}=\)

\begin{equation}
\label{eqn:TemplateSmbOrd}
\mathfrak{S}_{m-2,e-1}(\beta,\delta,\rho):=\left(X^{m-2},\ Y^{e-1}-\rho X^{m-3},\ XY-\rho\delta X^{m-3},\ T(X,Y)+\rho\beta X^{m-3}\right)
\end{equation}

\noindent
depends on \(\beta,\delta,\rho\), but is independent of \(\alpha,\gamma\).

\item
If \(k=1\), \(m\ge 6\), and \(e=3\), that is, if \(G\) is of coclass \(r=2\) and class \(c\ge 5\), then \(\mathfrak{A}=\)

\begin{equation}
\label{eqn:SmbOrdSec}
\mathfrak{S}_{m-2,2}(\beta,\delta,\rho):=\left(X^{m-2},\ Y^2-(1-\beta)\rho X^{m-3},\ XY-\rho\delta X^{m-3},\ T(X,Y)+\rho\beta X^{m-3}\right).
\end{equation}

\end{enumerate}

\end{theorem}

\begin{proof}
The statements are an immediate consequence of the Propositions
\ref{prp:LowPowX},
\ref{prp:LowPowY},
\ref{prp:LowPowYWithDefect},
and
\ref{prp:LowFurtherPolynomials}.
\end{proof}


\begin{remark}
\label{rmk:TemplateSmbOrd}
Theorem
\ref{thm:TemplateSmbOrd}
for \(\mathrm{cc}(G)\ge 2\)
is the Main Theorem of this article, since the ideal
\[\mathfrak{S}_{\mu,\nu}(\beta,\delta,\rho)
=\left(X^{\mu},\ Y^{\nu}-\rho X^{\mu-1},\ XY-\rho\delta X^{\mu-1},\ T(X,Y)+\rho\beta X^{\mu-1}\right)
\unlhd\mathbb{Z}\lbrack X,Y\rbrack,\]
with two parametrized exponents \(\mu\ge\nu\ge 1\) and three multiplicative parameters \(-1\le\beta,\delta,\rho\le 1\),
forms the \textit{general template} for all annihilators \(\mathfrak{A}=\mathfrak{A}(G)\)
which can arise from two-generated metabelian \(3\)-groups \(G\).
We briefly summarize in which way all the ideals
\(\mathfrak{R},\mathfrak{T},\mathfrak{V},\mathfrak{X},\mathfrak{Z},\mathfrak{Z}^\prime,\mathfrak{L}\),
which are well known since Scholz and Taussky
\cite[(6), p. 32]{SoTa},
and some new ideals \(\mathfrak{S},\mathfrak{U},\mathfrak{W},\mathfrak{Y}\)
can be obtained by \textit{specialization} from \(\mathfrak{S}_{\mu,\nu}(\beta,\delta,\rho)\).
We also indicate the transfer kernel type (TKT)
\cite{Ma2}
which is preferably associated with each ideal,
in particular, if \(G\) is realized as the second \(3\)-class group of a quadratic field
\cite{Ma4}.

\begin{enumerate}

\item
Groups of lower than second maximal class, i.e., of coclass \(\mathrm{cc}(G)\ge 3\):
\(\mu\ge\nu\ge 3\).

\begin{itemize}
\item
The template
\(\mathfrak{S}_{\mu,\nu}(\beta,\delta,\rho)
=\left(X^{\mu},Y^{\nu}-\rho X^{\mu-1},XY-\rho\delta X^{\mu-1},T(X,Y)+\rho\beta X^{\mu-1}\right)\)
itself occurs for TKT H.4, where \(\rho\ne 0\), \(\beta\ne 0\), \(\delta\ne 0\).
\item
The variant
\(\mathfrak{V}_{\mu,\nu}(\delta,\rho)
:=\left(X^{\mu},\ Y^{\nu}-\rho X^{\mu-1},\ XY-\rho\delta X^{\mu-1},\ T(X,Y)\right)\)
arises with TKT G.16, where \(\rho\ne 0\), \(\beta=0\), \(\delta\ne 0\).
\item
The variant
\(\mathfrak{T}_{\mu,\nu}(\beta,\rho)
:=\left(X^{\mu},\ Y^{\nu}-\rho X^{\mu-1},\ XY,\ T(X,Y)+\rho\beta X^{\mu-1}\right)\)
appears for TKT G.19, where \(\rho\ne 0\), \(\beta\ne 0\), \(\delta=0\).
\item
The new variant
\(\mathfrak{U}_{\mu,\nu}(\rho)
:=\left(X^{\mu},\ Y^{\nu}-\rho X^{\mu-1},\ XY,\ T(X,Y)\right)\)
is associated with TKT b.10, where \(\rho\ne 0\), \(\beta=0\), \(\delta=0\).
\item
The simplest variant
\(\mathfrak{R}_{\mu,\nu}
=\left(X^{\mu},\ Y^{\nu},\ XY,\ T(X,Y)\right)\), where \(\rho=0\),
occurs for many TKTs, F.7/11/12/13 and d.19/23/25.
\end{itemize}

\item
Groups of second maximal class, i.e., of coclass \(\mathrm{cc}(G)=2\):
\(\mu\ge\nu=2\).

\begin{itemize}
\item
\(\mathfrak{S}_{\mu,2}(\beta,\delta,\rho)
=\left(X^{\mu},Y^{2}-(1-\beta)\rho X^{\mu-1},XY-\rho\delta X^{\mu-1},T(X,Y)+\rho\beta X^{\mu-1}\right)\)
occurs for TKT H.4, where \(\rho\ne 0\), \(\beta\ne 0\), \(\delta\ne 0\).
\item
The variant
\(\mathfrak{Z}_{\mu}(\beta,\rho)
:=\left(X^{\mu},\ Y^{2}-(1-\beta)\rho X^{\mu-1},\ XY,\ T(X,Y)+\rho\beta X^{\mu-1}\right)\)
comes with TKT G.16, where \(\rho\ne 0\), \(\beta\ne 0\), \(\delta=0\).
\item
The variant
\(\mathfrak{Z}^\prime_{\mu}(\rho)
:=\left(X^{\mu},\ Y^{2}-\rho X^{\mu-1},\ XY,\ T(X,Y)\right)\)
is associated with TKT b.10, where \(\rho\ne 0\), \(\beta=0\), \(\delta=0\).
\item
The simplest variant
\(\mathfrak{X}_{\mu}
:=\left(X^{\mu},\ Y^{2},\ XY,\ T_3(X)\right)\), where \(\rho=0\),
occurs for many TKTs, E.6/8/9/14 and c.18/21.
Here, the trace polynomial degenerates to \(T_3(X)=X^2+3X+3\),
since \(Y^2,3Y\in\mathfrak{A}\) in \(T(X,Y)=X^2+3X+3+3Y+Y^2\).
\item
The special case
\(\mathfrak{L}_{2}
:=\left(X^{2},\ Y^{2},\ XY,\ 3\right)\), where \(\rho=0\), \(\mu=2\),
is associated with two TKTs, D.5/10.
The trace polynomial degenerates further to the constant \(3\),
since \(X^2,3X\in\mathfrak{A}\) in \(T_3(X)=X^2+3X+3\)
\end{itemize}

\item
Groups of maximal class, i.e., of coclass \(\mathrm{cc}(G)=1\):
\(\mu\ge\nu=1\).
Here, the parameter \(\gamma\) takes over the role of the parameter \(\rho\).

\begin{itemize}
\item 
\(\mathfrak{W}_{\mu}(\gamma)=\left(X^{\mu},\ Y-\gamma X^{\mu-1},\ T_3(X)\right)\)
comes with TKT a.1, where \(\gamma\ne 0\).
\item
\(\mathfrak{Y}_{\mu}=\left(X^{\mu},\ Y,\ T_3(X)\right)\)
appears for two TKTs a.2/3, where \(\gamma=0\).
\item
\(\mathfrak{Y}_{2}=\left(X^{2},\ Y,\ 3\right)\)
occurs for three TKTs a.2/3/3\({}^\ast\), where \(\gamma=0\), \(\mu=2\).
\item
\(\mathfrak{L}:=\mathfrak{Y}_{1}=\left(X,\ Y,\ 3\right)\)
arises with a single TKT A.1, where \(\gamma=0\), \(\mu=1\).
Here, spezialization stops,
since \(\mathfrak{L}\) is a \textit{maximal} ideal of the ring \(\mathbb{Z}\lbrack X,Y\rbrack\).
It is associated with the two extra special groups \(G\) of order \(27\).
One of them can be realized as the second \(3\)-class group of cyclic cubic fields
but not of quadratic fields.
\end{itemize}

\end{enumerate}

\noindent
Of course, one can get rid of the explicit parameters \(\beta,\delta,\rho\)
in the polynomial generators of the annihilator \(\mathfrak{A}\)
by putting \(\pm\) signs instead of nonzero values \(\rho\beta,\rho\delta,\rho\),
as it is done in
\cite{SoTa}.

\end{remark}



\subsection{The groups \(G\) with bicyclic centre \(\zeta_1(G)\simeq C_3\times C_3\)}
\label{ss:LowBic}

\noindent
In this section, we emphasize those groups \(G\)
whose annihilators \(\mathfrak{A}\) are most amenable to
determining the structure of the commutator subgroup \(G^\prime\) in \S\
\ref{s:StrComSbg}.

We begin with metabelian \(3\)-groups \(G\)
of second maximal class with CF-invariant \(e=3\).

\begin{theorem}
\label{thm:SmbOrdSecBic}

Let \(G\) be a metabelian \(3\)-group of coclass \(r=\mathrm{cc}(G)=2\), \(e=r+1=3\),
order \(\lvert G\rvert=3^n\ge 3^5\), nilpotency class \(c=\mathrm{cl}(G)=m-1\ge 3\), \(4\le m=n-1\),
and defect \(k=0\),
which is isomorphic to the representative \(G_\rho^{m,m+1}(\alpha,\beta,\gamma,\delta)\)
with parameters \(\rho=0\) and \(-1\le\alpha,\beta,\gamma,\delta\le 1\).

Then the annihilator \(\mathfrak{A}=\mathfrak{A}(G)\) of \(G\) is given by

\begin{equation}
\label{eqn:SmbOrdSecBic}
\mathfrak{A}=\mathfrak{X}_{m-2}=\mathfrak{R}_{m-2,2}=\left(X^{m-2},\ Y^2,\ XY,\ T_3(X)\right)
\end{equation}

\noindent
with the trace polynomial \(T_3(X)=X^2+3X+3\).

\end{theorem}

\begin{proof}
This follows immediately from item (1) in Theorem
\ref{thm:TemplateSmbOrd}
and item (2) in Remark
\ref{rmk:TemplateSmbOrd}.
\end{proof}


\noindent
The groups in Theorem
\ref{thm:SmbOrdSecBic},
which can be realized as second \(3\)-class groups of quadratic fields
\cite{Ma4},
have transfer kernel types \(\mathrm{D.5/10}\) for \(m=4\),
and generally \(\mathrm{E.6/8/9/14}\) or \(\mathrm{c.18/21}\) for \(m\ge 5\)
\cite{Ma2}.


\begin{remark}
\label{rmk:SmbOrdSecBic}
Theorem
\ref{thm:SmbOrdSecBic}
provides a new method for deriving the results by Scholz and Taussky
\cite{SoTa}
for the transfer kernel types \(\mathrm{D.5/10}\) and \(\mathrm{E.6/8/9/14}\)
\cite{Ma2}.
The novelty is the explicit connection of the exponent \(\mu\)
of the smallest power \(X^\mu\) in the annihilator \(\mathfrak{A}\),

\begin{equation}
\label{eqn:PolarizationMu}
\mu=m-2,
\end{equation}

\noindent
with the nilpotency index \(m\) of the group \(G\).
\end{remark}


\begin{remark}
\label{rmk:ClassFieldTheoryMu}
In the case of an arithmetical realization
\(G=\mathrm{Gal}(\mathrm{F}_3^2(K)\vert K)\)
of the group \(G\)
as the second \(3\)-class group of a number field \(K\), Formula
\eqref{eqn:PolarizationMu},
\(\mu=m-2\) admits
the determination of the exponent \(\mu\) in the annihilator \(\mathfrak{A}=\mathfrak{X}_{\mu}\),
if the orders \(3^{m-1}\), resp.  \((3^{m-1-k},3^{3},3^3,3^3)\), of the \(3\)-class groups
\(\mathrm{Cl}_3(\mathrm{F}_3^1(K))\) of the Hilbert \(3\)-class field \(\mathrm{F}_3^1(K)\),
resp. \(\mathrm{Cl}_3(L_i)\) of the four unramified cyclic cubic extensions \(L_i\), \(1\le i\le 4\),
of \(K\) are known
\cite[Thm. 3.12]{Ma4}.
If \(k=0\) is warranted, then computing the \(L_i\) is sufficient.
\end{remark}


\noindent
Next consider metabelian \(3\)-groups \(G\)
of lower than second maximal class with CF-invariant \(e\ge 4\).

\begin{theorem}
\label{thm:SmbOrdLowBic}

Let \(G\) be a metabelian \(3\)-group of coclass \(r=\mathrm{cc}(G)\ge 3\), \(e=r+1\ge 4\),
order \(\lvert G\rvert=3^n\ge 3^7\), nilpotency class \(c=\mathrm{cl}(G)=m-1\ge 4\), \(5\le m\le n-2\), \(n\le 2m-3\),
and defect \(k=0\),
which is isomorphic to the representative \(G_\rho^{m,n}(\alpha,\beta,\gamma,\delta)\)
with parameters \(\rho=0\) and \(-1\le\alpha,\beta,\gamma,\delta\le 1\).

Then the annihilator \(\mathfrak{A}=\mathfrak{A}(G)\) of \(G\) is given by

\begin{equation}
\label{eqn:SmbOrdLowBic}
\mathfrak{A}=\mathfrak{R}_{m-2,e-1}=\left(X^{m-2},\ Y^{e-1},\ XY,\ T(X,Y)\right)
\end{equation}

\noindent
with the trace polynomial \(T(X,Y)=T_3(X)+T_3(Y)-3=X^2+3X+3+3Y+Y^2\).

\end{theorem}

\begin{proof}
An immediate consequence of item (1) in Theorem
\ref{thm:TemplateSmbOrd}
and item (1) in Remark
\ref{rmk:TemplateSmbOrd}.
\end{proof}


\noindent
The groups in Theorem
\ref{thm:SmbOrdLowBic},
which can be realized as second \(3\)-class groups of quadratic fields
\cite{Ma4},
have the transfer kernel types \(\mathrm{F.7/11/12/13}\) or \(\mathrm{d.19/23/25}\)
\cite{Ma2}.


\begin{remark}
\label{rmk:SmbOrdLowBic}
Theorem
\ref{thm:SmbOrdLowBic}
provides a new method for deriving the results by Scholz and Taussky
\cite{SoTa}
for the transfer kernel types \(\mathrm{F.7/11/12/13}\)
\cite{Ma2}.
The novelty is the explicit connection of the exponents \(\mu\), resp. \(\nu\),
of the smallest power \(X^\mu\), resp. \(Y^\nu\), in the annihilator \(\mathfrak{A}\),

\begin{equation}
\label{eqn:PolarizationNu}
\mu=m-2 \quad \text{ and } \quad \nu=e-1,
\end{equation}

\noindent
with the nilpotency index \(m\), resp. the CF-invariant \(e\), of the group \(G\).
\end{remark}


\begin{remark}
\label{rmk:ClassFieldTheoryNu}
In the case of an arithmetical realization
\(G=\mathrm{Gal}(\mathrm{F}_3^2(K)\vert K)\)
of the group \(G\)
as the second \(3\)-class group of a number field \(K\), Formula
\eqref{eqn:PolarizationNu},
\(\mu=m-2\), \(\nu=e-1\), admits
the determination of the exponents \(\mu\) and \(\nu\) in the annihilator \(\mathfrak{A}=\mathfrak{R}_{\mu,\nu}\),
if the orders \(3^{m+e-4}\), resp.  \((3^{m-1-k},3^{e},3^3,3^3)\), of the \(3\)-class groups
\(\mathrm{Cl}_3(\mathrm{F}_3^1(K))\) of the Hilbert \(3\)-class field \(\mathrm{F}_3^1(K)\),
resp. \(\mathrm{Cl}_3(L_i)\) of the four unramified cyclic cubic extensions \(L_i\), \(1\le i\le 4\),
of \(K\) are known
\cite[Thm. 3.12]{Ma4}.
If \(k=0\) is warranted, then computing the \(L_i\) is sufficient.
\end{remark}



\subsection{The smallest groups \(G\) with cyclic centre \(\zeta_1(G)\simeq C_3\)}
\label{ss:LowSmlCyc}

\noindent
Up to now, we have not exploited the items (3) and (4) of Proposition
\ref{prp:LowFurtherPolynomials}
yet.
They concern exceptional groups with invariants \(m=5\), \(n=6\), \(\rho=\pm 1\),
whose relational parameters \(\beta,\delta,\rho\) were given in Table
\ref{tbl:SmlCycExpX}.

\begin{theorem}
\label{thm:SmbOrdSmlCyc}

The exceptional annihilators of the twelve metabelian \(3\)-groups \(G\)
of order \(\lvert G\rvert=3^6\), nilpotency class \(\mathrm{cl}(G)=m-1=4\), and defect \(k=1\),
are given in dependence on the relational parameters \(\beta,\delta,\rho\) by Table
\ref{tbl:SmbOrdSmlCyc}.

In particular, the annihilator of the four isomorphism classes of groups \(G\)
with relational parameters \(\beta=1\), \(\delta=1\), \(\rho=1\),
having the transfer kernel type \(\mathrm{H.4}\),
is given by
\[\mathfrak{A}=(X^2,Y^2,XY+3)\]
and the structure of the additive group underlying the residue class ring
\(\mathbb{Z}\lbrack X,Y\rbrack/\mathfrak{A}\)
with basis \((1,X,Y)\) is of type
\[(9,3,3).\]

\end{theorem}

\begin{proof}
Generally, this follows by inserting the relevant values of \(\beta,\delta,\rho\) in Table
\ref{tbl:SmbOrdSmlCyc}
into the polynomial generators of item (3) in Proposition
\ref{prp:LowFurtherPolynomials}.

For \(m=5\), \(n=6\), \(e=3\), and \(\beta=1\), \(\delta=1\), \(\rho=1\),
we can use item (4) in Proposition
\ref{prp:LowFurtherPolynomials}.
However, we exemplarily give a full demonstration:

\begin{enumerate}

\item
The supplementary connecting relation for \(s_4\) yields\\
\(s_2^{(x-1)^2}=s_4=\sigma_4^{-1}\sigma_5^{-1}\sigma_{m-1}^{\rho\beta}=\sigma_4^{\rho\beta-1}=1\),
since \(\sigma_5=\sigma_m=1\). Thus, we get \(X^2\in\mathfrak{A}\).

\item
From the supplementary connecting relation for \(t_4\),
together with the first central relation for \(\tau_{e+1}\), we obtain
\(s_2^{(y-1)^2}=t_4=\tau_4\tau_5\sigma_{m-1}^{\rho\beta}
=\tau_4\sigma_{m-1}^{\rho\beta}\),
since \(\tau_5=\tau_{e+2}=1\).\\
It follows that
\(s_2^{(y-1)^2}=\tau_{e+1}\sigma_{m-1}^{\rho\beta}
=\sigma_{m-1}^{-\rho}\sigma_{m-1}^{\rho\beta}=\sigma_{m-1}^{\rho(\beta-1)}=1\) and
\(Y^2\in\mathfrak{A}\).

\item
The relation for the third power of \(s_2\) yields\\
\(s_2^3=\sigma_4\sigma_{m-1}^{-\rho\beta}\tau_4^{-1}
=\sigma_{m-1}^{1-\rho\beta}\tau_{e+1}^{-1}=\sigma_{m-1}^{1-\rho\beta}\sigma_{m-1}^{\rho}
=\sigma_4^{1-\rho(\beta-1)}=\sigma_4\),\\
and consequently
\(s_2^9=\sigma_4^3=1\) and \(9\in\mathfrak{A}\).

\item
The first central relation for \(s_2^{(x-1)(y-1)}\) yields\\
\(s_2^{(x-1)(y-1)}=\sigma_{m-1}^{-\rho\delta}=\sigma_4^{-\rho\delta}=\sigma_4^{-1}
=s_2^{-3}\), and thus \(s_2^{(x-1)(y-1)+3}=1\) and \(XY+3\in\mathfrak{A}\).

\item
Using the inclusion \(X^2\in\mathfrak{A}\) and the power rule for commutators, we obtain\\
\(s_2^{3(x-1)}=s_2^{(x-1)^3+3(x-1)^2+3(x-1)}=s_2^{(x-1)((x-1)^2+3(x-1)+3)}=\lbrack s_2,x\rbrack^{(x-1)^2+3(x-1)+3}\)\\
\(=\lbrack s_2,x\rbrack^{1+x+x^2}=\lbrack s_2,x^3\rbrack=1\), since \(x^3\in G^\prime\).\\
Consequently, we have \(3X\in\mathfrak{A}\).
Similarly, it follows that \(3Y\in\mathfrak{A}\).

\end{enumerate}

\noindent
As polynomial generators of the annihilator \(\mathfrak{A}\) it suffices to take
\((X^2,Y^2,XY+3)\),\\
since
\(0\equiv X(XY+3)=X^2Y+3X\equiv 3X\), \quad
\(0\equiv Y(XY+3)=XY^2+3Y\equiv 3Y\),\\
and
\(0\equiv 3(XY+3)=3XY+9\equiv 9\pmod{\mathfrak{A}}\).\\
A basis of the additive group
\(\left(\mathbb{Z}\lbrack X,Y\rbrack/\mathfrak{A},+\right)\),
underlying the residue class ring of \(\mathbb{Z}\lbrack X,Y\rbrack\) modulo \(\mathfrak{A}\),
is given by
\((1,X,Y)\),
since all powers \(X^i\) and \(Y^i\), for \(i\ge 2\),
and thus all power products \(X^jY^\ell\), for \(j+\ell\ge 3\),
are contained in \(\mathfrak{A}\),
and \(XY\equiv -3\) is congruent to a multiple of \(1\).
The orders of the basis elements are given by
\(\mathrm{ord}(1)=9\),
\(\mathrm{ord}(X)=3\),
\(\mathrm{ord}(Y)=3\),
whence \(G^\prime\) is of type \((9,3,3)\).
\end{proof}


\begin{table}[ht]
\caption{Exceptional annihilators for \(m=5\), \(n=6\), \(e=3\), \(\rho=\pm 1\)}
\label{tbl:SmbOrdSmlCyc}
\begin{center}
\begin{tabular}{|c|l|c|rrr|c|}
\hline
 Groups                             & TKT  & Classes & \(\beta\) & \(\delta\) & \(\rho\) & \(\mathfrak{A}\)                  \\
\hline
 \(\langle 729,37\ldots 39\rangle\) & b.10 &   \(3\) &     \(0\) &      \(0\) &    \(1\) & \(\left(X^2-Y^2,\ XY,\ 9\right)\) \\
 \(\langle 729,44\ldots 47\rangle\) & H.4  &   \(4\) &     \(1\) &      \(1\) &    \(1\) & \(\left(X^2,\ Y^2,\ XY+3\right)\) \\
 \(\langle 729,56\ldots 57\rangle\) & G.19 &   \(2\) &    \(-1\) &      \(0\) &    \(1\) & \(\left(X^2-Y^2,\ XY,\ 3\right)\) \\
 \(\langle 729,34\ldots 36\rangle\) & b.10 &   \(3\) &     \(0\) &      \(0\) &   \(-1\) & \(\left(X^2+Y^2,\ XY,\ 3\right)\) \\
\hline
\end{tabular}
\end{center}
\end{table}



\section{Annihilators of metabelian \(p\)-groups of maximal class}
\label{s:MaxSmbOrdP}

\begin{theorem}
\label{thm:SmbOrdMaxP}

(Main Theorem for \(r=1\).)
For a metabelian \(p\)-group \(G\) of coclass \(r=\mathrm{cc}(G)=1\), nilpotency class \(c=\mathrm{cl}(G)=m-1\ge 2\),
that is, with nilpotency index \(m\ge 3\), and defect \(0\le k\le m-4\),
which is isomorphic to the representative \(G_a^m(z,w)\) with parameters \(0\le w,z\), \(a(m-1),\ldots,a(m-k)<p\),
where \(k=0\) for \(m\le 4\), the annihilator ideal \(\mathfrak{A}=\mathfrak{A}(G)\) of \(G\) is given by

\begin{equation}
\label{eqn:SmbOrdMaxPWithDef}
\mathfrak{A}=\mathfrak{W}_{m-2}(a):=\left(X^{m-2},\ Y-\sum_{\ell=1}^k\,a(m-\ell)X^{m-2-\ell},\ \sum_{\ell=1}^p\,\binom{p}{\ell}X^{\ell-1}\right).
\end{equation}

\noindent
In particular, if the defect of \(G\) is \(k=0\), that is, the family \quad \(a\) \quad of parameters is empty, then

\begin{equation}
\label{eqn:SmbOrdMaxP}
\mathfrak{A}=\mathfrak{Y}_{m-2}:=\mathfrak{W}_{m-2}(\emptyset)=\left(X^{m-2},\ Y,\ \sum_{\ell=1}^p\,\binom{p}{\ell}X^{\ell-1}\right).
\end{equation}

\end{theorem}

\begin{proof}
The proof is similar to the proof of Theorem
\ref{thm:SmbOrdMax3}.

According to the nilpotency relations in Formula
\eqref{eqn:NilpotencyRelMax3},
which are valid independently of the prime \(p\),
we have
\(1=s_m=s_2^{(x-1)^{m-2}}\) for \(j=m\),
but
\(1\ne s_{m-1}=s_2^{(x-1)^{m-3}}\) for \(j=m-1\).
As before, this implies that
\(X^{m-2}\in\mathfrak{A}\), but \(X^{m-3}\notin\mathfrak{A}\).

However, the commutator relation for \(s_2\) in Formula
\eqref{eqn:CommutatorRelMax3}
must now be replaced by Formula
\eqref{eqn:CommutatorRelMaxP},
and we obtain\\
\(s_2^{y-1}=\lbrack s_2,y\rbrack=\prod_{\ell=1}^k\,s_{m-\ell}^{a(m-\ell)}
=\prod_{\ell=1}^k\,s_2^{(x-1)^{m-\ell-2}\cdot a(m-\ell)}=s_2^{\sum_{\ell=1}^k\,a(m-\ell)(x-1)^{m-\ell-2}}\),
resp.\\
\(1=s_2^{(y-1)-\sum_{\ell=1}^k\,a(m-\ell)(x-1)^{m-\ell-2}}\),
which implies the inclusion
\(Y-\sum_{\ell=1}^k\,a(m-\ell)X^{m-\ell-2}\in\mathfrak{A}\).

The third power relation in Formula
\eqref{eqn:PowerRelMaxP}
for \(j=1\) yields\\
\(1=s_{1+1}^p\prod_{\ell=2}^p\,s_{1+\ell}^{\binom{p}{\ell}}=\prod_{\ell=1}^p\,s_{1+\ell}^{\binom{p}{\ell}}
=\prod_{\ell=1}^p\,s_2^{(x-1)^{\ell-1}\cdot\binom{p}{\ell}}
=s_2^{\sum_{\ell=1}^p\,\binom{p}{\ell}(x-1)^{\ell-1}}\).\\
According to the binomial formula, the exponent
\(\sum_{\ell=1}^p\,\binom{p}{\ell}(x-1)^{\ell-1}=\sum_{\ell=1}^p\,x^{\ell-1}=:\mathrm{tr}_p(x)\)
is equal to the \(p\)th \textit{trace element}.
Consequently, the pre-image \(T_p(X):=\sum_{\ell=1}^p\,\binom{p}{\ell}X^{\ell-1}\in\mathfrak{A}\)
of the trace element \(\mathrm{tr}_p(x)=\sum_{\ell=1}^p\,\binom{p}{\ell}(x-1)^{\ell-1}\)
under \(\psi\) lies in the annihilator.
\end{proof}

\begin{remark}
\label{rmk:MaxSmbOrdP}
The polynomial generators of our ideal \(\mathfrak{W}_{m-2}(a)\)
are briefly indicated by Miech in
\cite[p. 94]{Mi2}.
\end{remark}



\section{The structure of the commutator subgroup \(G^\prime\) of \(G\)}
\label{s:StrComSbg}

\begin{definition}
\label{dfn:NearlyHomocyclic}
Let \(p\ge 2\) be a prime number, and \(\mu\ge 1\) be a positive integer.
Assume that the unique quotient \(q\ge 0\) and remainder \(0\le r<p-1\)
of the Euclidean division of \(\mu\) by \(p-1\) are given by
\(\mu=q\cdot (p-1)+r\).

Then the \textit{nearly homocyclic} abelian \(p\)-group \(A(p,\mu)\) of order \(p^\mu\)
is defined by its type invariants

\begin{equation}
\label{eqn:NearlyHomocyclic}
\begin{cases}
\left(\overbrace{p^{q+1},\ldots,p^{q+1}}^{r \text{ times}},\overbrace{p^q,\ldots,p^q}^{(p-1)-r \text{ times}}\right) & \text{ if } \mu\ge p-1, \\
\left(\overbrace{p,\ldots,p}^{n \text{ times}}\right) & \text{ if } 1\le\mu<p-1.
\end{cases}
\end{equation}

\noindent
Additionally, let \(A(p,0):=1\) be the trivial group.

\end{definition}

\begin{remark}
\label{rmk:NearlyHomocyclic}
The \(p\)-rank of the nearly homocyclic abelian \(p\)-group \(A(p,\mu)\)
is \(p-1\), if \(\mu\ge p-1\), and \(\mu\), if \(0\le\mu<p-1\).
The group \(A(p,\mu)\) is \textit{strictly homocyclic} if and only if
either \(\mu\) is divisible by \(p-1\) or \(\mu<p-1\).

For \(p=3\) and \(\mu\ge 2\), we have
\(A(3,\mu)=\left(3^{\frac{\mu}{2}},3^{\frac{\mu}{2}}\right)\), if \(\mu=2q\) is even, and
\(A(3,\mu)=\left(3^{\frac{\mu+1}{2}},3^{\frac{\mu-1}{2}}\right)\), if \(\mu=2q+1\) is odd.

For \(p=2\) and any \(\mu\ge 0\), the group \(A(2,\mu)\) is cyclic of order \(2^\mu\).
\end{remark}



\subsection{\(G^\prime\) for metabelian \(p\)-groups \(G\) of maximal class}
\label{ss:ComSbgMaxP}

\noindent
Let \(p\) be a prime number and \(\mu\ge 1\) be an integer.
We first consider a rather general ideal \(\mathfrak{A}\) of bivariate polynomials in \(\mathbb{Z}\lbrack X,Y\rbrack\).

\begin{proposition}
\label{prp:OrdPowXMaxP}

Let \(\mathfrak{A}\unlhd\mathbb{Z}\lbrack X,Y\rbrack\) be an ideal
which contains \(X^\mu\) and the trace polynomial \(\sum_{\ell=1}^p\,\binom{p}{\ell}X^{\ell-1}\).
Then the order of the power \(X^{j}\) with exponent \(0\le j\le\mu-1\) is given by

\begin{equation}
\label{eqn:OrdPowXMaxP}
\mathrm{ord}(X^{j})=p^{q+1},
\end{equation}

\noindent
where \(q\ge 0\) and \(0\le r<p-1\) denote the unique quotient and remainder
of the Euclidean division \((\mu-1)-j=q\cdot (p-1)+r\) of \((\mu-1)-j\) by \(p-1\).

If \(Y\) can be expressed in terms of powers of \(X\) modulo \(\mathfrak{A}\),
then the system \(\left(1,X,X^2,\ldots,X^{p-2}\right)\), resp. \(\left(1,X,X^2,\ldots,X^{\mu-1}\right)\),
forms a basis of the additive group underlying the residue class ring \(\mathbb{Z}\lbrack X,Y\rbrack/\mathfrak{A}\)
if \(\mu\ge p-1\), resp. \(\mu<p-1\).

\end{proposition}

\begin{proof}
The trace polynomial
\[T_p(X):=\sum_{\ell=1}^p\,\binom{p}{\ell}X^{\ell-1}=\binom{p}{1}+\binom{p}{2}X+\binom{p}{3}X^{2}+\ldots+\binom{p}{p-1}X^{p-2}+X^{p-1}\]
is a monic polynomial of degree \(p-1\)
whose non-leading coefficients \(\binom{p}{\ell}\), with \(1\le\ell\le p-1\), are all divisible by the prime \(p\). 
The order of the powers \(X^{j}\) with exponents \(0\le j\le\mu-1\) is determined recursively.
We start with \(X^\mu\in\mathfrak{A}\) which is equivalent to \(X^\mu\equiv 0\pmod{\mathfrak{A}}\)
and thus with \(\mathrm{ord}(X^\mu)=1\).
For the recursion we multiply the congruence \(T_p(X)\equiv 0\pmod{\mathfrak{A}}\)
by \(p^{q}X^{j}\), when \(j=(\mu-1)-q(p-1)-r\) with \(q\ge 0\), \(0\le r\le p-2\),
and we obtain \(\mathrm{ord}(X^j)=p^{q+1}\).

The inclusion \(T_p(X)\in\mathfrak{A}\) can be interpreted as a congruence
\[X^{p-1}\equiv p+\binom{p}{2}X+\binom{p}{3}X^{2}+\ldots+pX^{p-2}\pmod{\mathfrak{A}},\]
which inductively shows that all powers \(X^j\) with exponents \(j\ge p-1\)
can be expressed by the system \(\left(1,X,X^2,\ldots,X^{p-2}\right)\).
The constraint on \(Y\) implies that this is also true for all powers \(Y^i\) with \(i\ge 1\)
and for the mixed powers \(X^jY^i\) with \(j,i\ge 1\).
\end{proof}


\noindent
With the aid of Proposition
\ref{prp:OrdPowXMaxP}
we are able to give a new proof of the structure of the derived subgroup \(G^\prime\)
of a metabelian \(p\)-group of maximal class,
which is included as a special case in Blackburn's Theorem 3.4
\cite[p. 68]{Bl}.
See also the Appendix in
\cite[p. 456]{Ma3}.

\begin{theorem}
\label{thm:ComSbgMaxP}

(Main Theorem for \(r=1\).)
Let \(G\) be a metabelian \(p\)-group of coclass \(r=\mathrm{cc}(G)=1\)
and nilpotency class \(c=\mathrm{cl}(G)=m-1\ge 2\), that is, with nilpotency index \(m\ge 3\),
which is isomorphic to the representative \(G_a^m(z,w)\) with parameters \(0\le a(m-1),\ldots,a(m-k),w,z\le p-1\).
Then the commutator subgroup \(G^\prime\) of \(G\)
is isomorphic to the nearly homocyclic abelian \(p\)-group of order \(p^{m-2}\),

\begin{equation}
\label{eqn:ComSbgMaxP}
G^\prime\simeq\mathbb{Z}\lbrack X,Y\rbrack/\mathfrak{A}\simeq A(p,m-2),
\end{equation}

\noindent
independently of whether \(k=0\), \(\mathfrak{A}=\mathfrak{Y}_{m-2}\), or \(k\ge 1\), \(\mathfrak{A}=\mathfrak{W}_{m-2}(a)\).

\end{theorem}

\begin{proof}
In Theorem
\ref{thm:SmbOrdMaxP},
we have seen that the annihilator \(\mathfrak{A}\) of a
metabelian \(p\)-group \(G\) of maximal class with nilpotency index \(m\ge 3\)
is either an ideal of the form \(\mathfrak{A}=\mathfrak{W}_{m-2}(a)\), if \(k\ge 1\),
or the ideal \(\mathfrak{A}=\mathfrak{Y}_{m-2}\), if \(k=0\).
Each of these ideals satisfies the assumptions of Proposition
\ref{prp:OrdPowXMaxP}
with \(\mu=m-2\ge 1\).
Thus, a basis of the additive group underlying the residue class ring \(\mathbb{Z}\lbrack X,Y\rbrack/\mathfrak{A}\)
is given by the system \(\left(1,X,X^2,\ldots,X^{p-2}\right)\), provided that \(\mu\ge p-1\).
The abelian type invariants of the derived subgroup
\(G^\prime\simeq\mathbb{Z}\lbrack X,Y\rbrack/\mathfrak{A}\)
are determined by the orders
\(\left(\mathrm{ord}(1),\mathrm{ord}(X),\mathrm{ord}(X^2),\ldots,\mathrm{ord}(X^{p-2})\right)\)
of the basis elements.
The maximal value of the order is achieved for the constant \(1=X^0\) with exponent \(j=0\),
the minimal value for the power \(X^{p-2}\) with exponent \(j=p-2\).
By Proposition
\ref{prp:OrdPowXMaxP},
these values are determined as follows.

We obtain \(\mathrm{ord}(1)=p^{Q+1}\) by the Euclidean division \((\mu-1)-j=(\mu-1)-0=\mu-1=Q(p-1)+R\),
which gives rise to a distinction of two cases, either \(R=p-2\), with maximal remainder, or \(0\le R<p-2\).
Accordingly, we have \(\mu=(Q+1)(p-1)\) a multiple of \(p-1\) in the first case,
where the other Euclidean division \((\mu-1)-j=(\mu-1)-(p-2)=\mu-p-1=Q(p-1)\)
yields the same quotient \(Q\) determining \(\mathrm{ord}(X^{p-2})=p^{Q+1}\).
Thus, \(G^\prime\) is strictly homocyclic with all \(p-1\) invariants equal to \(p^{Q+1}\).

In the second case, we have \(\mu=Q(p-1)+(R+1)\) with \(1\le R+1\le p-2\),
and the other Euclidean division \((\mu-1)-j=(\mu-1)-(p-2)=\mu-p-1=(Q-1)(p-1)+(R+1)\)
yields the smaller quotient \(Q-1\) determining \(\mathrm{ord}(X^{p-2})=p^{Q}\).
Thus, \(G^\prime\) is nearly homocyclic with \(R+1\) invariants equal to \(p^{Q+1}\)
but \((p-1)-(R+1)\) invariants equal to \(p^{Q}\).

In both cases, this structure coincides precisely with the nearly homocyclic abelian \(p\)-group
\(A(p,\mu)\) of order \(p^\mu\) with \(\mu=m-2\).
\end{proof}



\subsection{\(G^\prime\) for metabelian \(3\)-groups \(G\) of non-maximal class}
\label{ss:ComSbgLow}

\noindent
In each of the following Propositions
\ref{prp:OrdPowX}
--
\ref{prp:OrdOne},
we consider the general template for annihilator ideals,
\(\mathfrak{A}=\mathfrak{S}_{\mu,\nu}(\beta,\delta,\rho)\),
generated by the bivariate polynomials
\(X^{\mu}\), \(Y^{\nu}-\rho X^{\mu-1}\), \(XY-\rho\delta X^{\mu-1}\), \(T(X,Y)+\rho\beta X^{\mu-1}\),
where \(\mu\ge\nu\ge 2\), \(-1\le\beta,\delta,\rho\le 1\) and \(T(X,Y)=X^2+3X+3+3Y+Y^2\), as given in Theorem
\ref{thm:TemplateSmbOrd}.
Congruences are always viewed with respect to the modulus \(\mathfrak{A}\).

\begin{proposition}
\label{prp:OrdPowX}
If \(\mu\ge 3\), then
for each integer \(0\le j<\mu\),
the order of the power \(X^{\mu-j}\) of the indeterminate \(X\)
in the additive subgroup of the residue class ring \(\mathbb{Z}\lbrack X,Y\rbrack/\mathfrak{A}\)
is given by

\begin{equation}
\label{eqn:OrdPowX}
\mathrm{ord}(X^{\mu-j})=
\begin{cases}
3^{\frac{j}{2}}   & \text{ if } j \text{ is even},\\
3^{\frac{j+1}{2}} & \text{ if } j \text{ is odd}.
\end{cases}
\end{equation}

\end{proposition}

\begin{proof}
We start the induction with \(X^\mu\in\mathfrak{A}\), which is equivalent to
\(1\cdot X^\mu=X^\mu\equiv 0\), and yields
\(\mathrm{ord}(X^\mu)=1=3^{\frac{0}{2}}\) with even integer \(j=0\).\\
For the recursion, we use \(\rho\beta X^{\mu-1}+T(X,Y)\in\mathfrak{A}\),
which is equivalent with \(\rho\beta X^{\mu-1}+T(X,Y)\equiv 0\).
We assume that the assertion is true for some \(j\ge 0\).

\(\bullet\)
If \(1\le j+1<\mu\) is odd, we multiply the congruence by \(3^{\frac{j}{2}}X^{\mu-(j+1)}\) and get\\
\(0\equiv 3^{\frac{j}{2}}X^{\mu-(j+1)}\left(\rho\beta X^{\mu-1}+T(X,Y)\right)\)\\
\(=3^{\frac{j}{2}}\rho\beta X^{2\mu-(j+1)-1}+3^{\frac{j}{2}}X^{\mu-(j-1)}+3^{\frac{j+2}{2}}X^{\mu-j}+3^{\frac{j+2}{2}}X^{\mu-(j+1)}
+3^{\frac{j+2}{2}}X^{\mu-(j+1)}Y+3^{\frac{j}{2}}X^{\mu-(j+1)}Y^2\),\\
where \(X^{2\mu-(j+1)-1}\equiv 0\), since \(2\mu-(j+1)-1\ge\mu\),\\
\(3^{\frac{j}{2}}X^{\mu-(j-1)}+3^{\frac{j+2}{2}}X^{\mu-j}=(X+3)\cdot 3^{\frac{j}{2}}X^{\mu-j}\equiv 0\),
since \(j\) is even and \(\mathrm{ord}(X^{\mu-j})=3^{\frac{j}{2}}\),\\
\(3^{\frac{j+2}{2}}X^{\mu-(j+1)}Y\equiv 0\), since \(\mu-(j+1)\ge 1\), \(j+2\ge 2\), and \(\mathrm{ord}(XY)\) divides \(3\),\\
and \(X^{\mu-(j+1)}Y^2\equiv 0\), since \(\mu-(j+1)\ge 1\) and \(\mu\ge 3\).\\
Consequently, there only remains \(0\equiv 3^{\frac{j+2}{2}}X^{\mu-(j+1)}\), which means that
\(\mathrm{ord}(X^{\mu-(j+1)})=3^{\frac{(j+1)+1}{2}}\), as claimed for odd \(j+1\).

\(\bullet\)
If \(2\le j+1<\mu\) is even, we multiply the congruence by \(3^{\frac{j-1}{2}}X^{\mu-(j+1)}\) and get\\
\(0\equiv 3^{\frac{j-1}{2}}X^{\mu-(j+1)}\left(\rho\beta X^{\mu-1}+T(X,Y)\right)\)\\
\(=3^{\frac{j-1}{2}}\rho\beta X^{2\mu-(j+1)-1}+3^{\frac{j-1}{2}}X^{\mu-(j-1)}+3^{\frac{j+1}{2}}X^{\mu-j}+3^{\frac{j+1}{2}}X^{\mu-(j+1)}
+3^{\frac{j+1}{2}}X^{\mu-(j+1)}Y+3^{\frac{j-1}{2}}X^{\mu-(j+1)}Y^2\),\\
where \(X^{2\mu-(j+1)-1}\equiv 0\), since \(2\mu-(j+1)-1\ge\mu\),\\
\(3^{\frac{j-1}{2}}X^{\mu-(j-1)}\equiv 0\),
since \(j-1\) is even and \(\mathrm{ord}(X^{\mu-(j-1)})=3^{\frac{j-1}{2}}\),\\
\(3^{\frac{j+1}{2}}X^{\mu-j}\equiv 0\),
since \(j\) is odd and \(\mathrm{ord}(X^{\mu-j})=3^{\frac{j+1}{2}}\),\\
\(3^{\frac{j+1}{2}}X^{\mu-(j+1)}Y\equiv 0\), since \(\mu-(j+1)\ge 1\), \(j+1\ge 2\), and \(\mathrm{ord}(XY)\) divides \(3\),\\
and \(X^{\mu-(j+1)}Y^2\equiv 0\), since \(\mu-(j+1)\ge 1\) and \(\mu\ge 3\).\\
Consequently, there only remains \(0\equiv 3^{\frac{j+1}{2}}X^{\mu-(j+1)}\), which means that
\(\mathrm{ord}(X^{\mu-(j+1)})=3^{\frac{j+1}{2}}\), as claimed for even \(j+1\).

\(\bullet\)
The results are independent of the values of the parameters \(\beta,\delta,\rho\).
\end{proof}


\begin{proposition}
\label{prp:OrdPowY}
If \(\mu\ge\nu\ge 3\), then
for each integer \(0\le\ell<\nu\),
the order of the power \(Y^{\nu-\ell}\) of the indeterminate \(Y\)
in the additive subgroup of the residue class ring \(\mathbb{Z}\lbrack X,Y\rbrack/\mathfrak{A}\)
is given by

\begin{equation}
\label{eqn:OrdPowY}
\mathrm{ord}(Y^{\nu-\ell})=
\begin{cases}
3^{\frac{\ell}{2}} & \text{ if } \ell \text{ is even and } \rho=0,\\
3^{\frac{\ell+2}{2}} & \text{ if } \ell \text{ is even and } \rho=\pm 1,\\
3^{\frac{\ell+1}{2}} & \text{ if } \ell \text{ is odd}.
\end{cases}
\end{equation}

\end{proposition}

\begin{proof}
We start the induction with \(Y^\nu-\rho X^{\mu-1}\in\mathfrak{A}\), which is equivalent to
\(Y^\nu\equiv\rho X^{\mu-1}\), and yields
\(\mathrm{ord}(Y^\nu)=1=3^{\frac{0}{2}}\) for \(\rho=0\),
but \(\mathrm{ord}(Y^\nu)=\mathrm{ord}(X^{\mu-1})=3=3^{\frac{0+2}{2}}\) for \(\rho=\pm 1\),
both with even integer \(\ell=0\).\\
For the recursion, we use \(\rho\beta X^{\mu-1}+T(X,Y)\in\mathfrak{A}\),
which is equivalent to \(\rho\beta X^{\mu-1}+T(X,Y)\equiv 0\).
We assume that the assertion is true for some \(j\ge 0\).
Since the proof for \(\rho=0\) is similar to the proof of Proposition
\ref{prp:OrdPowX},
with interchanged roles of \(X\) and \(Y\),
we focus on the case \(\rho=\pm 1\).

\(\bullet\)
If \(1\le\ell+1<\nu\) is odd, we multiply the congruence by \(3^{\frac{\ell}{2}}Y^{\nu-(\ell+1)}\) and get\\
\(0\equiv 3^{\frac{\ell}{2}}Y^{\nu-(\ell+1)}\left(\rho\beta X^{\mu-1}+T(X,Y)\right)\)\\
\(=3^{\frac{\ell}{2}}\rho\beta X^{\mu-1}Y^{\nu-(\ell+1)}+3^{\frac{\ell}{2}}X^2Y^{\nu-(\ell+1)}+3^{\frac{\ell+2}{2}}XY^{\nu-(\ell+1)}
+3^{\frac{\ell+2}{2}}Y^{\nu-(\ell+1)}+3^{\frac{\ell+2}{2}}Y^{\nu-\ell}+3^{\frac{\ell}{2}}Y^{\nu-(\ell-1)}\),\\
where \(X^{\mu-1}Y^{\nu-(\ell+1)}\equiv 0\), since \(\mu-1\ge 2\), \(\nu-(\ell+1)\ge 1\),\\
\(X^2Y^{\nu-(\ell+1)}\equiv 0\), since \(\nu-(\ell+1)\ge 1\),\\
\(3^{\frac{\ell+2}{2}}XY^{\nu-(\ell+1)}\equiv 0\), since \(\ell+2\ge 2\), \(\nu-(\ell+1)\ge 1\), and \(\mathrm{ord}(XY)\mid 3\),\\
\(3^{\frac{\ell+2}{2}}Y^{\nu-\ell}\equiv 0\), since \(\ell\) is even and \(\mathrm{ord}(Y^{\nu-\ell})=3^{\frac{\ell+2}{2}}\),\\
and \(3^{\frac{\ell}{2}}Y^{\nu-(\ell-1)}\equiv 0\), since \(\ell-1\) is odd and \(\mathrm{ord}(Y^{\nu-(\ell-1)})=3^{\frac{\ell}{2}}\).\\
Consequently, there only remains \(0\equiv 3^{\frac{\ell+2}{2}}Y^{\nu-(\ell+1)}\), which means that
\(\mathrm{ord}(Y^{\nu-(\ell+1)})=3^{\frac{\ell+2}{2}}\), as claimed for odd \(\ell+1\).

\(\bullet\)
If \(2\le\ell+1<\nu\) is even, we multiply the congruence by \(3^{\frac{\ell+1}{2}}Y^{\nu-(\ell+1)}\) and get\\
\(0\equiv 3^{\frac{\ell+1}{2}}Y^{\nu-(\ell+1)}\left(\rho\beta X^{\mu-1}+T(X,Y)\right)\)\\
\(=3^{\frac{\ell+1}{2}}\rho\beta X^{\mu-1}Y^{\nu-(\ell+1)}+3^{\frac{\ell+1}{2}}X^2Y^{\nu-(\ell+1)}+3^{\frac{\ell+3}{2}}XY^{\nu-(\ell+1)}
+3^{\frac{\ell+3}{2}}Y^{\nu-(\ell+1)}+3^{\frac{\ell+3}{2}}Y^{\nu-\ell}+3^{\frac{\ell+1}{2}}Y^{\nu-(\ell-1)}\),\\
where \(X^{\mu-1}Y^{\nu-(\ell+1)}\equiv 0\), since \(\mu-1\ge 2\), \(\nu-(\ell+1)\ge 1\),\\ 
\(X^2Y^{\nu-(\ell+1)}\equiv 0\), since \(\nu-(\ell+1)\ge 1\),\\
\(3^{\frac{\ell+3}{2}}XY^{\nu-(\ell+1)}\equiv 0\), since \(\ell+3\ge 4\), \(\nu-(\ell+1)\ge 1\), and \(\mathrm{ord}(XY)\mid 3\),\\
and \(3^{\frac{\ell+3}{2}}Y^{\nu-\ell}+3^{\frac{\ell+1}{2}}Y^{\nu-(\ell-1)}=(3+Y)\cdot 3^{\frac{\ell+1}{2}}Y^{\nu-\ell}\equiv 0\),
since \(\ell\) is odd and \(\mathrm{ord}(Y^{\nu-\ell})=3^{\frac{\ell+1}{2}}\).\\
Consequently, there only remains \(0\equiv 3^{\frac{\ell+3}{2}}Y^{\nu-(\ell+1)}\), which means that
\(\mathrm{ord}(Y^{\nu-(\ell+1)})=3^{\frac{(\ell+1)+2}{2}}\), as claimed for even \(\ell+1\).

\(\bullet\)
The results for even integers \(\ell\) depend on the value of the parameter \(\rho\).
All results are independent of the values of the parameters \(\beta,\delta\).
\end{proof}


\begin{proposition}
\label{prp:OrdOne}
If \(\mu\ge\nu\ge 3\) and \(\rho=0\), then the order of the constant \(1\)
in the additive subgroup of the residue class ring \(\mathbb{Z}\lbrack X,Y\rbrack/\mathfrak{A}\)
is given by

\begin{equation}
\label{eqn:OrdOne}
\mathrm{ord}(1)=
\begin{cases}
3^{\frac{\mu}{2}}   & \text{ if } \mu \text{ is even},\\
3^{\frac{\mu+1}{2}} & \text{ if } \mu \text{ is odd}.
\end{cases}
\end{equation}

\end{proposition}

\noindent
Note that formula
\eqref{eqn:OrdOne}
can be viewed as a particular instance of formula
\eqref{eqn:OrdPowX},
formally, but it requires a different proof.

\begin{proof}
By the Propositions
\ref{prp:OrdPowX}
and
\ref{prp:OrdPowY},
we know \(\mathrm{ord}(X^{\mu-1})=3\) and the exact orders of the powers with exponents
\(1=\mu-(\mu-1)=\nu-(\nu-1)\) and \(2=\mu-(\mu-2)=\nu-(\nu-2)\).

\begin{equation}
\label{eqn:OrdSmallPow}
\begin{aligned}
\mathrm{ord}(X)=
\begin{cases}
3^{\frac{\mu}{2}}   & \text{ if } \mu \text{ is even},\\
3^{\frac{\mu-1}{2}} & \text{ if } \mu \text{ is odd}.
\end{cases}
\quad & \quad
\mathrm{ord}(X^{2})=
\begin{cases}
3^{\frac{\mu-2}{2}} & \text{ if } \mu \text{ is even},\\
3^{\frac{\mu-1}{2}} & \text{ if } \mu \text{ is odd}.
\end{cases}
\\
\mathrm{ord}(Y)=
\begin{cases}
3^{\frac{\nu}{2}}   & \text{ if } \nu \text{ is even},\\
3^{\frac{\nu-1}{2}} & \text{ if } \nu \text{ is odd}.
\end{cases}
\quad & \quad
\mathrm{ord}(Y^{2})=
\begin{cases}
3^{\frac{\nu-2}{2}} & \text{ if } \nu \text{ is even},\\
3^{\frac{\nu-1}{2}} & \text{ if } \nu \text{ is odd}.
\end{cases}
\end{aligned}
\end{equation}

\(\bullet\)
If \(\mu\) is even, we multiply the congruence \(\rho\beta X^{\mu-1}+T(X,Y)\equiv 0\) by \(3^{\frac{\mu-2}{2}}\) and get\\
\(0\equiv 3^{\frac{\mu-2}{2}}\left(\rho\beta X^{\mu-1}+T(X,Y)\right)\)\\
\(=3^{\frac{\mu-2}{2}}\rho\beta X^{\mu-1}+3^{\frac{\mu-2}{2}}X^2+3^{\frac{\mu}{2}}X
+3^{\frac{\mu}{2}}\cdot 1+3^{\frac{\mu}{2}}Y+3^{\frac{\mu-2}{2}}Y^2\),\\
Since \(\mu\ge\nu\ge 4\), there only remains \(0\equiv 3^{\frac{\mu}{2}}\cdot 1\), i.e.,
\(\mathrm{ord}(1)=3^{\frac{\mu}{2}}\), as claimed for even \(\mu\).

\(\bullet\)
If \(\mu\) is odd, we multiply the congruence \(\rho\beta X^{\mu-1}+T(X,Y)\equiv 0\) by \(3^{\frac{\mu-1}{2}}\) and get\\
\(0\equiv 3^{\frac{\mu-1}{2}}\left(\rho\beta X^{\mu-1}+T(X,Y)\right)\)\\
\(=3^{\frac{\mu-1}{2}}\rho\beta X^{\mu-1}+3^{\frac{\mu-1}{2}}X^2+3^{\frac{\mu+1}{2}}X
+3^{\frac{\mu+1}{2}}\cdot 1+3^{\frac{\mu+1}{2}}Y+3^{\frac{\mu-1}{2}}Y^2\),\\
Since \(\mu\ge\nu\ge 3\), there only remains \(0\equiv 3^{\frac{\mu+1}{2}}\cdot 1\), i.e.,
\(\mathrm{ord}(1)=3^{\frac{\mu+1}{2}}\), as claimed for odd \(\mu\).
\end{proof}



\begin{theorem}
\label{thm:ComSbgSecBic}

Let \(G\) be a metabelian \(3\)-group of coclass \(r=\mathrm{cc}(G)=2\), \(e=r+1=3\),
order \(\lvert G\rvert=3^n\ge 3^5\), nilpotency class \(c=\mathrm{cl}(G)=m-1\ge 3\), \(4\le m=n-1\),
and defect \(k=0\).

Then the commutator subgroup \(G^\prime\) of \(G\) is isomorphic to the direct sum
 
\begin{equation}
\label{eqn:ComSbgSecBic}
G^\prime\simeq\mathbb{Z}\lbrack X,Y\rbrack/\mathfrak{A}\simeq A(3,m-2)\oplus\mathrm{C}(3)
\end{equation}

\noindent
of the nearly homocyclic abelian \(3\)-group of order \(3^{m-2}\)
with the cyclic group of order \(3\).
Therefore, \(G^\prime\) is precisely of \(3\)-rank three.

\end{theorem}

\noindent
Note that the nearly homocyclic group
\[A(3,m-2)\quad\text{ is of type }\quad
\begin{cases}
\left(3^{\frac{m-2}{2}},3^{\frac{m-2}{2}}\right), &\text{ if } m\ge 4 \text{ is even},\\
\left(3^{\frac{m-1}{2}},3^{\frac{m-3}{2}}\right), &\text{ if } m\ge 5 \text{ is odd}.
\end{cases}\]

\begin{proof}
The additive group underlying the residue class ring
\(\mathbb{Z}\lbrack X,Y\rbrack/\mathfrak{A}\)
with basis \((1,X,Y)\) is isomorphic to the direct sum
\(A(3,m-2)\oplus\mathrm{C}(3)\)
with abelian type invariants\\
\(\left(\mathrm{ord}(1),\mathrm{ord}(X),\mathrm{ord}(Y)\right)\).
\end{proof}


\begin{theorem}
\label{thm:ComSbgLowBic}

Let \(G\) be a metabelian \(3\)-group of coclass \(r=\mathrm{cc}(G)\ge 3\), \(e=r+1\ge 4\),
order \(\lvert G\rvert=3^n\ge 3^7\), nilpotency class \(c=\mathrm{cl}(G)=m-1\ge 4\), \(5\le m<n-1\),
and defect \(k=0\).

Then the commutator subgroup \(G^\prime\) of \(G\) is isomorphic to the direct sum

\begin{equation}
\label{eqn:ComSbgLowBic}
G^\prime\simeq\mathbb{Z}\lbrack X,Y\rbrack/\mathfrak{A}\simeq A(3,m-2)\oplus A(3,e-2)
\end{equation}

\noindent
of two nearly homocyclic abelian \(3\)-groups of orders \(3^{m-2}\) and \(3^{e-2}\).
Therefore, \(G^\prime\) is non-elementary of \(3\)-rank exactly four.

\end{theorem}

\noindent
Concerning the two nearly homocyclic groups arising as direct summands in Formula
\eqref{eqn:ComSbgLowBic},
note that
 
\begin{eqnarray*}
\label{eqn:LowBic}
\mathrm{A}(3,m-2)&\text{ is of type }&
\begin{cases}
\left(3^{\frac{m-2}{2}},3^{\frac{m-2}{2}}\right),&\text{ if }m\ge 6\text{ is even},\\
\left(3^{\frac{m-1}{2}},3^{\frac{m-3}{2}}\right),&\text{ if }m\ge 5\text{ is odd},
\end{cases}\\
\mathrm{A}(3,e-2)&\text{ is of type }&
\begin{cases}
\left(3^{\frac{e-2}{2}},3^{\frac{e-2}{2}}\right),&\text{ if }e\ge 4\text{ is even},\\
\left(3^{\frac{e-1}{2}},3^{\frac{e-3}{2}}\right),&\text{ if }e\ge 5\text{ is odd}.
\end{cases}
\end{eqnarray*}

\begin{proof}
The additive group underlying the residue class ring
\(\mathbb{Z}\lbrack X,Y\rbrack/\mathfrak{A}\)
with basis \((1,X,Y,Y^2)\) is isomorphic to the direct sum
\(A(3,m-2)\oplus A(3,e-2)\)
with abelian type invariants\\
\(\left(\mathrm{ord}(1),\mathrm{ord}(X),\mathrm{ord}(Y),\mathrm{ord}(Y^2)\right)\).
\end{proof}


\begin{theorem}
\label{thm:ComSbgSmlCyc}

The exceptional commutator subgroups of the twelve metabelian \(3\)-groups \(G\)
of order \(\lvert G\rvert=3^6\), nilpotency class \(\mathrm{cl}(G)=m-1=4\), and defect \(k=1\),
are given in dependence on the relational parameters \(\beta,\delta,\rho\) by Table
\ref{tbl:ComSbgSmlCyc}.

\end{theorem}

\begin{proof}
This is a consequence of Theorem
\ref{thm:SmbOrdSmlCyc}.
\end{proof}

\begin{table}[ht]
\caption{Exceptional commutator subgroups for \(m=5\), \(n=6\), \(e=3\), \(\rho=\pm 1\)}
\label{tbl:ComSbgSmlCyc}
\begin{center}
\begin{tabular}{|c|l|c|rrr|c|c|}
\hline
 Groups                             & TKT  & Classes & \(\beta\) & \(\delta\) & \(\rho\) & Basis           & \(G^\prime\)  \\
\hline
 \(\langle 729,37\ldots 39\rangle\) & b.10 &   \(3\) &     \(0\) &      \(0\) &    \(1\) &     \((1,X,Y)\) &   \((9,3,3)\) \\
 \(\langle 729,44\ldots 47\rangle\) & H.4  &   \(4\) &     \(1\) &      \(1\) &    \(1\) &     \((1,X,Y)\) &   \((9,3,3)\) \\
 \(\langle 729,56\ldots 57\rangle\) & G.19 &   \(2\) &    \(-1\) &      \(0\) &    \(1\) & \((1,X,Y,Y^2)\) & \((3,3,3,3)\) \\
 \(\langle 729,34\ldots 36\rangle\) & b.10 &   \(3\) &     \(0\) &      \(0\) &   \(-1\) & \((1,X,Y,Y^2)\) & \((3,3,3,3)\) \\
\hline
\end{tabular}
\end{center}
\end{table}



\section{Final remark}
\label{s:Final}

This article is dedicated to the memory of the
gifted young group theorist Otto Schreier (\(1901\)--\(1929\)),
whose early passing away in the age of \(28\)
was a serious deprivation of the entire mathematical community,
since his extensive knowledge comprised
all the most up to date branches of mathematics in his time, for instance,
class field theory, algebraic and analytic number theory,
topology, functions of several complex variables,
finite groups, discontinuous group actions, and representations of continuous groups.
An obituary has been written by his friend and colleague Menger
\cite{Me}.
The central concept of the present article,
annihilator ideals of finite metabelian \(p\)-groups,
goes back to Furtw\"angler
\cite{Fw},
who was one of the academic advisors of Schreier at the University of Vienna,
and was developed further by Schreier
\cite{Sr1,Sr2},
Scholz and Taussky
\cite{SoTa},
Szekeres
\cite{Sz1,Sz2},
Brink
\cite{Br},
Brink and Gold
\cite{BrGo},
and Brown and Porter
\cite{BwPt}.



\section{Acknowledgements}
\label{s:Thanks}

The author gratefully acknowledges that his research is supported financially
by the Austrian Science Fund (FWF): P 26008-N25.




\end{document}